\documentclass[12pt]{amsart}

\usepackage{amsfonts,amsmath,amsxtra,amssymb,amsthm,tikz}
\usepackage{times}
\usepackage{anysize}
\usepackage{graphicx}

\input xy
\xyoption{all}
\usepackage[all,knot]{xy}

\marginsize{1in}{1in}{1in}{1in}

\numberwithin{equation}{subsection}

\theoremstyle{plain}
\newtheorem{theorem}[equation]{Theorem}
\newtheorem{lemma}[equation]{Lemma}
\newtheorem{proposition}[equation]{Proposition}
\newtheorem{corollary}[equation]{Corollary}

\theoremstyle{definition}
\newtheorem{definition}[equation]{Definition}
\newtheorem{example}[equation]{Example}
\newtheorem{remark}[equation]{Remark}

\DeclareMathOperator{\codim}{codim}
\DeclareMathOperator{\Imm}{Im}
\DeclareMathOperator{\gr}{gr}

\DeclareMathOperator{\Ord}{ord}
\DeclareMathOperator{\HFL}{HFL}

\DeclareMathOperator{\HL}{HL}
\DeclareMathOperator{\bdeg}{bdeg}
\DeclareMathOperator{\rank}{rank}
\DeclareMathOperator{\Gr}{Gr}

\newcommand{\Aa}{\mathcal{A}}
\newcommand{\Cc}{A^{-}}\newcommand{\Ccc}{\mathcal C}
\newcommand{\Ee}{\mathcal{E}}
\newcommand{\Kk}{\mathcal{K}}
\newcommand{\Ii}{\mathcal{I}}
\newcommand{\Jj}{\mathcal{J}}
\newcommand{\Ll}{\mathcal{L}}
\newcommand{\Hh}{\mathcal{H}}
\newcommand{\Oo}{\mathcal{O}}
\newcommand{\Pp}{\mathcal{P}}
\newcommand{\Ss}{\mathcal{S}}
\newcommand{\gbar}{\overline{g}}
\newcommand{\BC}{\mathbb{C}}
\newcommand{\BF}{\mathbb{F}}
\newcommand{\BL}{\mathbb{L}}
\newcommand{\BP}{\mathbb{P}}
\newcommand{\BR}{\mathbb{R}}
\newcommand{\BZ}{\mathbb{Z}}
\newcommand{\frakv}{\mathfrak{v}}
\newcommand{\sq}{\square}
\newcommand{\bsq}{\blacksquare}

\author{Eugene Gorsky \and Andr\'as N\'emethi}

\address{
Eugene Gorsky,
Mathematics Department, Columbia University\\
2990 Broadway, New York, NY 10027 USA}
\address {Department of Mathematics, UC Davis, One Shields Ave\\ Davis, CA 95616}
\address {International Laboratory of Representation Theory and Mathematical Physics\\
NRU-HSE, 7 Vavilova St.\\
Moscow, Russia 117312}
\email{egorsky@math.columbia.edu}

\address{
Andr\'as N\'emethi,
 Alfr\'ed R\'enyi Institute of Mathematics,
Hungarian Academy of Sciences\\
Re\'altanoda utca 13-15, H-1053, Budapest, Hungary}
\email{nemethi.andras@renyi.mta.hu }

\title{Lattice and Heegaard Floer homologies of algebraic links}

\begin{document}

\maketitle

\begin{abstract}
We compute the Heegaard Floer link homology of algebraic links in terms of the multivariate Hilbert function of the corresponding plane curve singularities.
The main result of the paper identifies four homologies:
(a)  the
Heegaard Floer link homology of the local embedded link,
(b) the lattice homology associated with the Hilbert function, (c) the  homologies of the
projectivized complements of local hyperplane arrangements cut out from the local algebra,
and
(d) a generalized version of the Orlik--Solomon algebra of these local arrangements.   In particular, the Poincar\'e
polynomials of all these homology groups are the same, and we also show that they agree with the
coefficients of the motivic Poincar\'e  series of the singularity.



\end{abstract}

\section{Introduction}

Complex analytic/algebraic plane curve singularities provide interesting connections between
analytic theory of  singularities and low dimensional topology, in particular, knot theory.
The rigidity properties of algebraic links help to compute the topological invariants
via analytic methods, while knot theory provides topological characterizations
for certain analytic
invariants (see e.g. \cite{book,EN,Milnor} and references therein).
E.g., in \cite{cdg} Campillo, Delgado and Gusein-Zade related the multi-variable Alexander polynomial of an
algebraic link to the multi-dimensional semigroup of the divisors of analytic functions.
They also identified the coefficients of the Alexander polynomial
with the Euler characteristics of certain
 projectivized hyperplane arrangement  complements associated with the ring of functions.
In this paper, we prove a ``homological lift'' of their theorem by {\it identifying
the Heegaard Floer link homology of the local analytic  link with
the homology of these hyperplane arrangements, and providing a concrete and computable
description of them in terms of classical singularity invariants of algebraic links (Hilbert
function, or Alexander polynomial)}.

Usually, the identification of the Heegaard Floer link homology $HFL^-$ is very hard, and very few concrete
examples are known. For $L$--space links we propose a   strategy,
which makes a conceptual simplification,  however at this generality this strategy
is also obstructed seriously at several points.
The strategy provides  a spectral sequence converging to  $HFL^-$, whose $E_2$ term is a
lattice cohomology associated with certain weights, which are determined by the
Alexander polynomial. 
But for a general $L$--space link the collapse of the spectral sequence is not guaranteed.

However, for algebraic links we eliminate all these obstructions as follows. Firstly, in \cite{gn}
we proved that algebraic links are $L$--space links, hence the strategy runs.
Then, we identify the $HFL$--weights
needed for the $E_2$  (lattice cohomology) term with the values of the
Hilbert  function of the local algebra (where the multi--filtration is given by valuations induced by the
normalization).
For this we need an `analytic inversion' formula, which provides the Hilbert function
from the Alexander polynomial.

This Hilbert function is the central singularity invariant,
it has a rich structure which will be exploited deeply. Based on this,
we analyze the properties
of the lattice cohomology (defined in \cite{Nemethi08} in a very general setup)
associated with the Hilbert function weights, and we show that
it is isomorphic to the cohomology of certain hyperplane arrangements embedded in the
ring of functions. For this step we need to use and improve the Orlik--Solomon theory
of the cohomology of hyperplane arrangement complements.
The next step exploits the structure of Orlik--Solomon cohomology rings
(determined by the rigid matroid properties of our Hilbert function). We define a
bigrading on the Orlik--Solomon complex and prove a vanishing result which guarantees that the cohomology
is supported on a line (with respect to this bigrading).
This intrinsic structure and vanishing
will imply finally the collapse of the above mentioned spectral sequence involving  the $HFL^-$ theory
(showing that all the higher differentials endowed with the bigrading are
necessarily trivial).

The final  picture identifies the ranks of the following four graded homologies:
\begin{itemize}
\item[(a)] The Heegaard Floer link homology of the local embedded link of the germ,
\item[(b)] The local lattice homology associated with the Hilbert function,
\item[(c)] The (simplicial) homologies of the projectivized complements of local hyperplane arrangements cut out from the local algebra by
valuations given by the  normalizations of irreducible components,
\item[(d)] A generalized  version of the Orlik--Solomon algebra of these local arrangements.

\end{itemize}
In particular, the Poincar\'e polynomials of all these homology groups are the same, and we also show that they agree with the
coefficients of the motivic Poincar\'e  series of the singularity germ \cite{cdg3,intfun,mozu}.
Since the homologies have no $\BZ$-torsion, the corresponding
Poincar\'e polynomials provide the complete description of the corresponding  homologies.

 It is important to mention that  the above isomorphisms are defined separately for each
 Alexander grading, which belongs to the lattice $\BZ^r$
(where $r$ is the number of components of a link). For each lattice point $v\in \BZ^r$ we define a separate topological space $\Hh(v)$ (which is either empty or a complement to a hyperplane arrangement) and relate its homology to $\HFL^-(v)$. This recovers  $\HFL^-=\bigoplus_{v}\HFL^-(v)$ as a $\BZ^r\oplus\BZ$--graded vector space (for a comment
regarding coefficients, see Remark \ref{remark:warning}). 
Here the last $\BZ$-grading is the homological grading.
(All other homologies in the above list (a)--(d) are graded similarly.) 

Some of the important structures present in $\HFL^-$ are not immediately recovered with this approach. In particular, the Heegaard Floer theory defines operators $U_1,\ldots,U_r$ which act on $\HFL^-$ and shift the Alexander grading in various directions. It seems plausible that the action of $U_i$ is determined by the Hilbert function too, but we do not  study this action in the present paper -- such a study would require a comparison of spaces $\Hh(v)$ for different $v$.

In order to realize the above program, we need to recall/improve several properties
of Heegaard Floer link homology of $L$--space links in section \ref{s:La} and of
local algebraic curve singularities (e.g. how to invert the Alexander polynomial  to the
Hilbert function) (section \ref{s:Hil}), to develop the theory of lattice cohomology (associated with the Hilbert function weights) (section \ref{s:lattice}), and to adjust and improve the theory of Orlik--Solomon algebras (section \ref{s:Arran}).
Based on all these we finish the main proof in  section \ref{s:L}. Finally, in section \ref{NewExample} we explicitly compute the homologies (a)-(d)
for the Hopf link, corresponding to the singularity $\{xy=0\}$.

\subsection{}
The next subsections provide more details on the involved  invariants and identifications (for the precise definitions
and statements see
the next sections).

Trough the paper the following notations will be used.
The number of link components will be denoted by $r$.
 Set $K_0=\{1,\ldots,r\}$.
Let $e_i$ denote the $i$-th coordinate vector in $\BZ^r$.
For a subset $K\subset K_0$ we
write $e_{K}=\sum_{i\in K}e_{i}$ and $e=e_{K_0}=\sum e_i$.
Given $v\in \BZ^r$, we define $v_{K}=\sum_{i\in K}v_ie_i$.
$|K|$ denotes the cardinality of $K$.
We set a partial order on $\BZ^r$ by
$$
u\preceq v\ \Longleftrightarrow \ u_i\le v_i \ \ \ \mbox{for all $i$}.
$$
\subsection{The Hilbert series and the related singularity `package'}
Let $(C,0)=(\cup_{i=1}^{r}C_i,0)$ be a reduced plane curve singularity at the origin in $\BC^2$, where
 $C_i$ are the irreducible components. Let $\gamma_i:(\BC,0)\rightarrow (C_i,0)$ be the
  normalization of the components.
We consider  $r$ valuations on the $\BC$--algebra
$\Oo=\Oo_{\BC^2,0}$ defined by
$
\frakv_i(f)=\Ord\left(f\left(\gamma_i(t)\right)\right)
$,
and a $\BZ^r$-indexed filtration
$$
J(v)=\{f\in \Oo\ |\ \frakv_{i}(f)\ge v_i\ \mbox{for all $i$}\}.
$$
The Hilbert function $h:\BZ^r\to \BZ$ is defined by  $h(v)=\dim \Oo/J(v)$, while the multivariable
Hilbert series by $H(t)=\sum_{v} h(v)t_1^{v_1}\cdots t_r^{v_r}$, cf.  \ref{def:Hil}.  It guides most of the classical
analytic and topological invariants of the germ. For example, the multivariable Poincar\'e series
satisfies  $P(t)=-H(t)\cdot\prod_i(1-t^{-1}_i) $. By
\cite{cdg} $P(t)$ is related to  the
multivariable Alexander polynomial $\Delta(t)$ as follows: $\Delta(t) =P(t)$ if $r>1$,
 while  $\Delta(t)=(1-t)P(t)$ for $r=1$.
 This shows that $\Delta(t)$ is determined by the Hilbert series $H(t)$.
 We prove an `Inversion Theorem' \ref{reconst} providing an explicit way to recover $H(t)$ from $\Delta(t)$.
 (This explicit formula  can be used to define an analogue of $H(t)$ for any non-algebraic link as well; this
 plays an important role in the study of $L$-space links in Heegaard Floer link theory:
 it produces the weights of the lattice complex whose lattice cohomology is the
 $E_2$ term of the spectral sequence, cf. Theorem \ref{thm: l space}.)

Another objects determined by the valuations are the topological spaces
$$
\Hh(v):=\{f\in \Oo\ |\ \frakv_i(f)=v_i\ \mbox{\rm for all}\ i\}
$$
and their projectivizations $\BP\Hh(v)$. Although $\Hh(v)$ and
$\BP\Hh(v)$ are infinite-dimensional,
 they can be projected onto finite-dimensional varieties with affine fibers.
Furthermore, $\Hh(v)=J(v)\setminus \cup_i J(v+e_i)$ (where $e_i$ are the base vectors), hence
 $\Hh(v)$ is either empty or a complement of  a central hyperplane arrangement,
see section \ref{ss:ARR}.  It turns out that the Euler characteristic of $\BP\Hh(v)$ is exactly
the coefficient $\pi_v$ of $t^v$ in the Poincar\'e series $P(t)$. Replacing the Euler characteristic $\pi_v$
by the Poincar\'e polynomial $\pi_v(q)$
of the homology of  $\BP\Hh(v)$ (or by the class of $\BP\Hh(v)$ in the Grothendieck ring of algebraic
varieties), we obtain the `motivic Poincar\'e series' $\Pp(t;q)=\sum_v \pi_v(q)t^v$ \cite{cdg3,intfun,mozu}.

\subsection{The Orlik--Solomon theory} To describe the homology of $\Hh(v)$, we
need some facts from the theory of hyperplane arrangements.
Let $\{H_1,\ldots, H_r\}$ be a collection of hyperplanes in a complex vector space $V$. Brieskorn in \cite{bries} proved  that the de Rham cohomology of the complement $\Hh:=V\setminus \cup_{i} H_{i}$ is generated as an algebra by the classes of 1-forms $z_i=\frac{d\ell_i}{\ell_i}$, where $\ell_i$ are the defining equations of $H_i$. Orlik and Solomon \cite{or1} gave an explicit combinatorial description of the ideal of relations between $z_i$ in terms of
linear dependencies between $\ell_i$ (see  \ref{ss:4.2}).
To connect the Orlik-Solomon theory with the $\BZ$--module structure of the lattice and $HFL^-$ cohomologies,
we prove  the following improvement of their result (see Theorem \ref{homology of d0}).

\begin{theorem}
Consider the free anticommutative algebra $\Ee$ generated by $z_i$. It is naturally bigraded: a monomial $\wedge_{i\in K}z_i$ has bidegree
$(|K|,\rho(K))$, where $\rho(K):=\dim V/\cap_{i\in K}H_{i}$.
\begin{itemize}
\item[(a)] There is a differential $\partial_0$ on $\Ee$ of bidegree $(-1,0)$ such that $H_{*}(\Ee,\partial_0)\simeq H_{*}(\Hh)$.
\item[(b)] There is a differential $\partial_U=\partial_0+U\partial_1$ on $\Ee[U]$ such that $H_{*}(\Ee[U],\partial_U)\simeq H_{*}(\BP\Hh)$.
\item[(c)] All classes in the homology of $\partial_U$ have $U$-degree $0$ and can be presented as sums of
    monomials $\alpha=\wedge_{i\in K}z_i$ such that the hyperplanes
    $H_{i\in K}$  are independent.
\end{itemize}
\end{theorem}

\begin{corollary}
\label{bidegree collapse} (a)
The homology of $\partial_0$ or $\partial_U$ inherits a bidegree, and for the nontrivial
generators  $|\alpha|=\rho(\alpha)$. Therefore,
the bidegrees in non-trivial  homology elements lie on a line.

(b) The $U$--action on $H_{*}(\Ee[U],\partial_U)\simeq H_{*}(\BP\Hh)$ is trivial.
\end{corollary}

\subsection{The lattice homology}\label{ss:LatHom}
This note introduces the lattice homology of $(C,0)$.
Recall that in \cite{Nemethi08}
the lattice homology of a normal surface singularity was introduced via
the lattice provided by its  resolution graph (or plumbing graph of the link).
That invariant created a bridge between the analytic  invariants of the surface
singularity and  several topological invariants (like  Seiberg--Witten invariant and Heegaard Floer homology) of its 3--dimensional link.
The goal of the present construction is similar; nevertheless here we rely on the lattice $\BZ^r$
discussed above, and
the needed weight function is provided by the normalization of $C$, namely
by  the Hilbert function $h(v)$.

In short, the definition  for an arbitrary weight function $w:\BZ^r\to \BZ$ runs as follows.
The lattice complex $\Ll^{-}_w$
is generated over $\BZ[U]$ by cubes $\sq$ of all dimensions in
$\BR^r$, with vertices in the lattice $\BZ^r$.
For  such a cube we define $w(\sq)=\max_{x\in \sq\cap{\BZ}^r}w(x)$.
The differential is defined as
$$\partial_{U}(U^m\sq)=U^m\cdot \sum_{i}\varepsilon_i U^{w(\sq)-w(\sq_i)}\sq_i,$$
where $\sq_i$ are the
oriented boundary cubes of $\sq$, and $\varepsilon_i$ are the corresponding signs
(as in the boundary operator of the classical cubic homology). We define the
{\it homological degree } of the generators by
$\deg (U^m\sq)=-2m+\dim (\sq)-2w(\sq)$;
$\partial_U$ decreases it by one.

The complex $\Ll^{-}_w$ is naturally $\BZ^{r}$-filtered:
the subcomplex $\Ll^{-}_w(v)$ is generated by the cubes contained in
the positive quadrant originating at $v$. One of our main theorems describes
the homology of the subcomplexes $\Ll^{-}_w(v)$ and the associated graded complexes $\gr_v \Ll^{-}_w$
for all $v$.

\begin{theorem}
\label{homology of L}
\begin{enumerate}
\item[(a)] If $w$ is non-decreasing (that is, $w(v)\leq w(u)$ for $v\preceq u$), then
 $H^{*}(\Ll^{-}_w(v))\simeq \BZ[U]$ with a generator of homological degree $-2w(v)$.
\item[(b)] In the algebraic case (that is, if $w=h$), the Poincar\'e polynomial $P_{\gr_v \Ll_h^{-}}(t)$
of the homology of $\gr_v \Ll^{-}_h$ agrees with the $v$--coefficient in the motivic Poincar\'e series:
$$P_{\gr_v \Ll^{-}_h}(-t^{-1})=t^{h(v)}\pi_{v}(t).$$
\item[(c)] The following (co)homologies are isomorphic:
$$H_{-2h(v)-*}(\gr_v\Ll^-_h)\simeq
H^*(\BP\Hh(v)),$$
and both spaces are free $\BZ$-modules.
\item[(d)] The induced $U$--action on $H_{*}(\gr_v\Ll^-_h)$ is trivial.
\end{enumerate}
\end{theorem}

We prove the parts of this theorem in Theorems \ref{homology of l(u)}, \ref{zlat} and \ref{th:Lat+OS}.
\subsection{Heegaard Floer link homology}
We relate the Heegaard Floer link homology  $\HFL^{-}$ of an algebraic link
to lattice homology of the corresponding plane curve singularity.
(For the definition of $\HFL^{-}$ see \cite{os,os2,os3,os4}  and  \cite{ras}).

Recall that an $L$-space is a 3-manifold with minimal possible rank of its Heegaard Floer homology,
and an $L$--space link is a link in $S^3$ such that a sufficiently large surgery of $S^3$
along its components yield an $L$--space.
Ozsv\'ath and Szab\'o proved in \cite{os} that the Heegaard Floer  homology of an $L$--space {\it knot} is determined by its Alexander polynomial.
 Hedden proved in \cite{hedden} than every algebraic {\it knot} is an $L$--space knot. As a consequence,  Heegaard Floer  homology of an algebraic knot is determined by its Alexander polynomial.
However, for $L$--space {\it links} is not known if their Heegaard Floer link homology is determined by the
multivariable Alexander polynomial.

As a generalization of the above  facts valid for knots,
we propose the following program.  First, in
 \cite{gn} (motivated by the present manuscript), the authors observe
 that all algebraic links are $L$-space links. Then, by the general theory of Ozsv\'ath and Szab\'o
 of $L$--space links and by `Large Surgery Theorem'
  of Manolescu and Ozsv\'ath \cite{mo}, such a link $L\subset S^3$
  provides a function $g:\BZ^r\to \BZ$ as follows (we call it
 $HFL$--weight function).  The $HFL^-$ complex is a $\BZ[U_1,\ldots,U_r]$
  module with Alexander filtration $\{\Cc(v)\}_{v\in\BZ^r}$,
  where $U_i(\Cc(v))\subset \Cc(v+e_1)$. $\Cc(v) $ is a subcomplex and a $\BZ[U_1,\ldots,U_r]$ submodule.
 Its homology is a free rank one $\BZ[U]$-module (with $U=U_1$). Then   $g(v)$  is essentially the homological
 degree of its unique generator  (similarly to Theorem \ref{homology of L}(a)).  The function $g(v)$ is
 determined by the multi-variable Alexander polynomial  of $L$ (Theorem  \ref{th:NEWTH}).
We prove the following  (see Theorem \ref{thm: l space ss}).
\begin{theorem}
\label{thm: l space}
Let $L$ be an $L$--space link and  let $g:\BZ^r\to \BZ$ be its $HFL$--weight
function. Then for each fixed $v\in \BZ^r$ there exists a spectral
sequence with the following properties:
\begin{itemize}
\item[(a)]The $E^1$ page can be identified (as a $\BZ[U]$ module)
with the lattice complex $\gr_{v}\Ll_{g}^{-}$ associated with $g(v)$.
\item[(b)]The $E^2$ page is isomorphic (over $\BZ$)
to the local lattice homology associated with  $g(v)$.
\item[(c)]The $E^{\infty}$ page is isomorphic (as graded $\BZ$-module, where the grading is the homological one)
to $\HFL^{-}(L,v)$,
the Heegaard Floer link homology of $L$ with Alexander grading $v$.
Moreover,
the spectral sequence collapses at $E^r$ page (or earlier).
\item[(d)] If $r\leq 3$  then the spectral sequence collapses at the $E^2$ page.
\end{itemize}
\end{theorem}
For algebraic links the following additional facts hold
(cf. Theorems \ref{g equals h} and \ref {thm:alg l space}).

\begin{theorem}\label{last}
 If $L$ is the link of a plane curve singularity $(C,0)$ then the
$HFL$--weight function $g(v)$
coincides with the Hilbert function $h(v)$. Moreover,
the spectral sequence always collapses  at the $E^2$ page.
\end{theorem}
\begin{corollary}
If\, $L$ is the algebraic link of $(C,0)$ then for each fixed  $v\in \BZ^r$ one has
$$\HFL^{-}(L,v)\simeq H_{*}(\gr_{v} \Ll^{-}_h)\simeq H^{-2h(v)-*}(\BP\Hh(v))$$
(isomorphism of graded $\BZ$ modules). Moreover,
the Poincar\'e polynomial of the Heegaard Floer link homology is described by
Theorem \ref{homology of L} by the coefficients of the motivic Poincar\'e series.
Furthermore, $\HFL^{-}(L,v)\neq 0$ if and only if $v$ belongs to the semigroup of $(C,0)$.
\end{corollary}

Theorem \ref{thm: l space} can be compared with \cite[Theorem 1.1]{OSZSt2},
where a similar spectral sequence from a different form of lattice homology
(associated with a plumbing graph) to Heegaard Floer homology was considered.
 For the first part of  Theorem \ref{last} we
use the `Inversion Theorem' \ref{reconst}, and in the proof
of collapse we use some specific properties of the
  lattice homology and Orlik--Solomon algebras
  established in Theorem \ref{zlat} (cf. Corollary \ref{bidegree collapse}).

\section*{Acknowledgements}

The authors are grateful to S. Gusein-Zade, J. Fern\'andez de Bobadilla, M. Hedden, C. Manolescu, J. Moyano-Fern\'andez, P. Ozsv\'ath, J. Rasmussen and Z. Szab\'o for the useful remarks and discussions.
We also thank the Oberwolfach Mathematical Institute, where part of this work was done, for the hospitality. The article was finalized during the visit of A. N. at Princeton University.
The research of E. G. was partially supported by the grants RFBR-10-01-678, RFBR-13-01-00755, NSh-8462.2010.1, NSF grant DMS-1403560 and the Simons foundation.
A. N.  is partially supported by OTKA Grants 81203, 100796 and 112735.

\section{Heegaard Floer link homology}\label{s:La}

\subsection{Review of Heegaard Floer link homology}\
In this subsection we recall certain basic algebraic structures of Heegaard Floer link homology. For more
see \cite{mo,os,os2,os3,os4,ras}.

To every 3-manifold $M$ with fixed Heegaard splitting one can associate a {\em Heegaard Floer complex} $CF^{-}(M)$
of free $\BZ[U]$-modules. The operator $U$ has homological degree $(-2)$, and the differential $d$ has degree $(-1)$. This complex is not unique, but different choices
(e.g. of a splitting)
lead to quasi-isomorphic complexes. Therefore the homology of $CF^{-}(M)$ is an invariant of $M$ called
 {\em Heegaard Floer homology} and denoted by $HF^{-}(M)$. For example, $HF^{-}(S^3)=\BZ[U]$.

To a link $L=L_1\cup\cdots\cup L_{r}\subset S^3$ one can associate a $\BZ^r$-filtered complex of
$\BZ[U_1,\ldots,U_r]$-modules, denoted by $CFL^{-}(L)$.  If one ignores the filtration, then the complex is quasi-isomorphic
to the Heegaard Floer complex  $CF^{-}(S^3)$, where all the operators $U_i$ are homotopic to each other,
cf.  \cite{os3}. One can also
consider this complex as a $\BZ[U]$-module, where $U=U_1$.

However, the filtration (called Alexander filtration) captures nontrivial information about the link.
For $v\in \BZ^r$, we will denote the Alexander filtration by $\{\Cc(v)\}_v$.
Each  $\Cc(v)=(\oplus_\nu A^{-,\nu}(v),d)$
is a subcomplex of $CFL^{-}(L)$
(in \cite{mo} they are denoted  by $\mathfrak{A}^{-}(v)$)
\footnote{For a more transparent
  match with the algebraic picture, we reverse the sign of $v$, thus reversing the direction of the
  filtration as well.}.
 The upper index  $\nu$ denotes the homological (or Maslov) grading.
They satisfy
\begin{equation}\label{eq:INCL}\begin{array}{l}
\Cc(v)\supset  \Cc(u) \mbox{  \ for    $u\succ v$, \ and }\\
\Cc(v)\cap \Cc(u)= \Cc(\max\{v,u\}).
                               \end{array}
\end{equation}
 The subcomplexes $\Cc(v)$ are $\BZ[U_1,\ldots,U_r]$-submodules,
 the operators $U_i$ have homological degree $(-2)$ and are
homotopic to each other.  Moreover, $U_i(\Cc(v))\subset \Cc(v+e_i)$.

The Heegaard Floer link homology is defined as the homology of the associated graded pieces of $\Cc(v)$:
$$
\HFL^{-}(L,v):=H_{*}(\, (\gr\Cc)(v)\,), \ \ \mbox{where} \ (\gr\Cc)(v):=
\Cc(v)/\sum_{u\succ v}\Cc(u).
$$
For example, for $r=1$ one has $\HFL^{-}(L,v)=H_{*}(\Cc(v)/\Cc(v+1)).$

\begin{remark}\label{remark:warning}
At present, Heegaard Floer link homology is defined only for $\BF_2$ coefficients, hence, strictly speaking, all results of this section and the last section are valid only over $\BF_2$. Nevertheless, we believe that all the statements
are true over $\BZ$ as well, but the cautious reader might take everywhere $\BF_2$ instead of $\BZ$.
\end{remark}
By \cite[Proposition 9.2]{os3}, the Euler characteristic of the Heegaard Floer link homology
coincides with the Reidemeister torsion, and it satisfies
\begin{equation}\sum_{v\in \BZ^r}\
\chi(HFL^-(L,v))\cdot t^v=\begin{cases}
\Delta(t)& \text{if}\ r>1,\\
\frac{\Delta(t)}{1-t}& \text{if}\ r=1,
\end{cases}\end{equation}
where $\Delta$ is the multivariable Alexander polynomial of $L$.


\subsection{$L$-space links}

In \cite{os} Ozsv\'ath and Szab\'o  introduced the notion of an $L$--space:  a rational homology 3--sphere
$M$ is an $L$--space if for any ${\rm spin}^c$--structure $s$ one has $\rank \widehat{HF}(M,s)=1$
(or, equivalently, $HF^-(M,s)$ is a free $\BZ[U]$--module of rank 1).

\begin{definition}
A link $L\subset S^3$ is called an {\it $L$-space link}
 if a sufficiently large surgery on all of its components is an $L$-space.
\end{definition}

The following `Large Surgery Theorem'  shows the importance of the $L$-space property.

\begin{theorem}(\cite[Theorem 10.1]{mo}, see also \cite[Lemma 4.2]{OSZSt2})
\label{th:lspace}
If $d_1,\ldots,d_r$ are sufficiently large integers, then the homology of $\Cc(v)$ (considered as $\BZ[U]$-module)
is isomorphic to the Heegaard Floer homology $HF^-$ of the 3-manifold
$S^3_{d_1,\ldots,d_r}(L)$ obtained from $S^3$ by $d_i$-surgery along the components of the link $L_i$
(for a certain ${\rm spin}^c$--structure depending on $v$).

In particular, if $L\subset S^3$ is an $L$-space link, then for any $v\in \BZ^r$ the homology of\,
$\Cc(v)$ is a free $\BZ[U]$--module of rank 1.
\end{theorem}

Let $\gbar(v)$ denote the homological degree of the unique generator in $H_{*}(\Cc(v))$.

\begin{lemma}
\label{lem:gbar}
For all $i$ and $v\in \BZ^r$ either $\gbar(v+e_i)=\gbar(v)$ or $\gbar(v+e_i)=\gbar(v)-2$.
Furthermore, the inclusion map $\Cc(v+e_i)\hookrightarrow \Cc(v)$ induces an injection on homology.
\end{lemma}

\begin{proof}
 One has the following inclusions:
\begin{equation}\label{eq:Inc}
\Cc(v)\supset \Cc(v+e_i)\supset U_i\Cc(v)\supset U_i\Cc(v+e_i).
\end{equation}
By Theorem \ref{th:lspace}, $H_{*}(\Cc(v)/U_i\Cc(v))$ and $H_{*}(\Cc(v+e_i)/U_i\Cc(v+e_i))$ are free $\BZ$-modules of
rank 1 with generators of homological degrees $\gbar(v)$ and $\gbar(v+e_i).$
Similarly to \cite[Lemma 3.2]{os}  (see also \cite{intfun}), from (\ref{eq:Inc}) one obtains the following alternative:
\begin{equation}
\label{alternative}
\begin{cases}
\gbar(v+e_i)=\gbar(v)\ \text{and}\  \dim H_{*}(\Cc(v)/\Cc(v+e_i))=0, \ \ \text{or}\\
\gbar(v+e_i)=\gbar(v)-2\  \text{and}\  \dim H_{*}(\Cc(v)/\Cc(v+e_i))=1.\\
\end{cases}
\end{equation}
The long exact sequence in the homology implies the injectivity of the inclusion.
\end{proof}

Motivated by Theorem \ref{homology of l(u)} (valid for algebraic links)
  we introduce the following definition.
\begin{definition}
We define the {\em $HFL$--weight function} of an $L$-space link by
$$g(v):=-{\textstyle{\frac{1}{2}}}\,\gbar(v).$$
\end{definition}

Note that by Lemma \ref{lem:gbar} the values of $\gbar(v)$ have the same parity for all $v$, hence
$g(v)\in\BZ$ or $g(v)\in\textstyle{\frac{1}{2}}+\BZ$ for all $v$, hence $g(v+e_i)-g(v)\in\{0,1\}$.

\begin{corollary}
\label{cor:injective}
For all $u\succeq v$ the inclusion $i_{uv}:\Cc(u)\hookrightarrow \Cc(v)$ induces an injective map on homology.
If $z(u)$ and $z(v)$ are generators in $H_{*}(\Cc(u))$ and in $H_{*}(\Cc(v))$ respectively, then
$$
i_{uv,*}(z(u))=U^{g(u)-g(v)}z(v).
$$
\end{corollary}

\begin{definition}
Consider the ``iterated cone'' complex
$$\Kk(v):=\bigoplus_{K\subset \{1,\ldots,r\}}\Cc(v+e_K),\ \ D=d+\sum_{i=1}^{r}\epsilon_{i,K}\partial_i,$$
where $d$ is a differential on $\Cc$,  $\partial_{i}:\Cc(v+e_K)\to \Cc(v+e_K-e_i)$ is the inclusion map ($i\in K$),
and the signs  $\epsilon_{i,K}=\pm 1$ are chosen such that $D^2=0$.
\end{definition}

It is useful to present $\Kk(v)$ as an $r$-dimensional cube with the
complexes $\{\Cc(v+e_K)\}_K$ at vertices. The differential $d$ acts in vertices, while $\partial:=\sum_{i=1}^{r}\epsilon_{i,K}\partial_i$
acts along the edges. The homological grading of a generator  $x\in A^{-,\nu}(v+e_K)$, considered as a generator in $\Kk(v)$, equals
$|K|+\nu$. The differential $d$ decreases $\nu$ by 1 and preserves $|K|$, the differential $\partial$ decreases $|K|$ by
 1 and preserves $\nu$,  so both decrease the total grading by 1.
\begin{lemma}
\label{lem: iterated cone}
The complexes $(\gr \Cc)(v)$ and $\Kk(v)$ are quasi-isomorphic.
\end{lemma}

\begin{proof}
We prove this by induction on $r$. For $r=1$ it is clear that $\Kk(v)$ is just the cone of the inclusion map $\Cc(v+e_1)\to \Cc(v)$,
so it is quasi-isomorphic to $(\gr \Cc)(v)=\Cc(v)/\Cc(v+e_1)$.

For $r>1$, we can write (using (\ref{eq:INCL}))
$$(\gr \Cc)(v)=\Big(\Cc(v)/\sum_{i=1}^{r-1} \Cc(v+e_i))\Big)\Big/\Big(\Cc(v+e_r)/\sum_{i=1}^{r-1} \Cc(v+e_i+e_r))\Big).$$
Each of these quotients can be realized as an iterated cone, and $(\gr \Cc)(v)$
can be realized as a cone of the natural map between them.
\end{proof}

The following theorem and its proof is similar to the main result of \cite{OSZSt2},
although it appears in a different setup.
The algebraic construction of the `iterated cone' complex $\Kk$ can be compared
 with the construction appearing in \cite[Theorem 4.3]{OSZSt2}.

For the definition of the lattice complex and cohomology see subsection
\ref{ss:3.1} and section \ref{s:lattice}.
\begin{theorem}
\label{thm: l space ss}
Let $L$ be an $L$-space link with $r$ components. Let us fix a point $v\in \BZ^r$.
There exists a spectral sequence with the following properties:
\begin{itemize}
\item[a) ]Its $E^2$ page is isomorphic (as graded $\BZ$ module)
to $H_{*}(\gr_v \Ll^{-}_{g})$, where $\Ll^{-}_{g}$ denotes the
lattice complex associated with the $HFL$--weight function $g(v)$.
\item[b)] Its $E^{\infty}$ page is isomorphic (as graded $\BZ$ module)
to $\HFL^{-}(L,v)$, the Heegaard Floer link homology of $L$
with Alexander grading $v$.
\item[c)] The spectral sequence collapses at $E^{r}$ page (or earlier).
\item[d)] If $L$ has three or less components, then the  spectral sequence collapses  at $E^2$.
\end{itemize}
\end{theorem}

\begin{proof}
One has two (anti)commuting differentials $d$ and $\partial$ on the complex $\Kk(v)$, hence there
exists a spectral sequence which starts
with the cohomology of $d$ and converges to the cohomology of $D=d+\partial$.
By Lemma \ref{lem: iterated cone}, its (multigraded) $E^{\infty}$ page is isomorphic to
$$
E^{\infty}(v)=H_{*}(\Kk(v),D)=H_{*}((\gr \Cc)(v))=HFL^{-}(L,v).
$$
On the other hand, by Theorem \ref{th:lspace}, the $E^{1}$ page of this spectral sequence is isomorphic to
$$
E^1(v)=H_{*}(\Kk(v),d)=\bigoplus_{K}H_{*}(\Cc(v+e_K))=\bigoplus_{K}\BZ[U]\cdot z(v+e_K),
$$
where, as above, we denote the generator in the homology of $\Cc(u)$ by $z(u)$.
One can naturally identify this $E^1$ page with the lattice complex $(\gr_v\Ll_{g}^{-},\gr \partial_{U})$, via the identification of
 $z(v+e_K)$ by $\square(v,K)$.
Note that the $\nu$--grading  of $z(v+e_A)$ (in $\Kk(v)$) equals $\gbar(v+e_K)=-2g(v+e_K)$, hence
the homological grading of $U^{m}z(v+e_K)$ equals
$
\nu(U^{m}z(v+e_K))+|K|=-2m-2g(v+e_K)+|K|,
$
in agreement with the definition of the homological degree in the lattice complex
in \ref{ss:3.1}, see also   \eqref{eq:HOMDEG1}.
 The next differential is induced by $\partial$, and by Corollary \ref{cor:injective} it agrees with the lattice differential
for the weight function $g(v)$.
Indeed,
$$
\partial(z(v+e_K))=\sum_{i\in K}\pm\partial_{i}(z(v+e_K))=\sum_{i\in K}\pm
U^{g(v+e_K)-g(v+e_K-e_i)}z(v+e_K-e_i).
$$

The differential $d_k$ in the spectral sequence decreases $|K|$  by $k$ and
increases the $\nu$-grading (homological grading in vertices of the cube) by $k-1$
(assuming that $d=d_0$ and $\partial=d_1$). In particular, for $k>r$ the differential $d_k$ vanishes automatically.
Moreover, the class of the unique $r$-dimensional cube is not in the kernel of the lattice differential,
so $d_r$ vanishes too.

Since the $\nu$-gradings of all classes on $E^1$ page has the same parity, $d_k$ can be nontrivial only if $k$ is odd.
In particular, for $r\le 3$ we have $d_2=d_3=0$, so $E^2=E^\infty$.
 \end{proof}

The next theorem expresses  the function $g(v)$ in terms of the Euler characteristic  of
the $HFL^-$ homology (or, equivalently, in terms of multivariable Alexander polynomial).


%

\begin{theorem}\label{th:NEWTH} Let $L_K$ denote the sublink associated with $K\subset K_0$.
Then for every $v\in \BZ^r$
$$
g(v)=\sum_{K\subset K_0}(-1)^{|K|-1}\sum_{0\preceq u\preceq v_K-e_{K}}\chi (HFL^-(L_K,u)).$$
\end{theorem}
\begin{proof}
Since the Heegaard Floer complex is finitely generated as $\BZ[U]$-module, there exists 
 $N=(N_1,\ldots,N_r)$ large enough such that $\Cc(v)\subset \Cc(-N)$ for any $v$. Hence
 $$g(v)=g(\max\{v,-N\}).$$
For a subset $K=\{i_1,\ldots,i_{|K|}\}\subset K_0$ consider a sublink $L_K:=\cup_{i\in K} L_{i}$.
 $L_K$ is also $L$--space link (cf. \cite[Lemma 1.6]{Liu}),
 so it defines a $HFL$-weight function $g_K$ on the sublattice of $\BZ^r$ supported on $K$.
 By \cite[Proposition 7.1]{os3}, the restriction of the filtration $\Cc(v)$ to this sublattice coincides with
the filtration on the Heegaard Floer complex for the sublink $L_K$.  Given $v_{i_1},\ldots,v_{i_{|K|}}$, define
$$
u(v_{i_1},\ldots,v_{i_{|K|}}):=\begin{cases}
v_j,\ & j\in K,\\
-N_j,\ & j\notin K,\\
\end{cases}
$$
then $\Cc(u(v_{i_1},\ldots,v_{i_{|K|}}))\simeq \Cc_{L_K}(v_{i_1},\ldots,v_{i_{|K|}})$ and
\begin{equation}
\label{restriction for g}
g(u(v_{i_1},\ldots,v_{i_{|K|}}))=g_K(v_{i_1},\ldots,v_{i_{|K|}}).
\end{equation}
At Euler characteristic level we obtain
 \begin{equation}
 \label{Euler characteristic for g}
\chi(HFL^-(L_{K},v))=\chi\Big(\Cc(v)/\sum_{i\in K}\Cc(v+e_i)\Big)=\sum_{M\subset K}(-1)^{|M|-1}g(v+e_{M}).
 \end{equation}
This is a linear system of equations for $g(v)$, and by
Theorem \ref{reconst} (where  $0$ should be replaced by $-N$) the function $g(v)$ is defined
uniquely by the equations \eqref{restriction for g}  and \eqref{Euler characteristic for g}  up to an overall shift.
\end{proof}


\section{The Hilbert function and its relation with other  invariants }\label{s:Hil}

In this section we discuss the connections between the multi-variable Alexander polynomial,
three series (Poincar\'e,   Hilbert and  motivic Poincar\'e),  and the semigroup associated with an
 isolated  plane curve  singularity. All the statements, except those which involve
 the Alexander polynomial, are valid for arbitrary (non necessarily plane) curve
 singularity germs. The Alexander polynomial, by its very essence, is an invariant of the embedded
 topological type (hence of the embedded link); in the algebraic case it
 connects the theory of links
 of $S^3$ with the above  algebraic invariants.

\subsection{The Hilbert series of the multi-index filtration}
We fix a local reduced
plane curve singularity with $r$ irreducible components $C_i$
and normalizations  $\gamma_i:(\BC,0)\to (C_i,0)$.
Set the valuations
$\frakv_i(f)=\Ord_t\left(f\left(\gamma_i(t)\right)\right)$ on  $\Oo=\Oo_{\BC^2,0}$,
and a $\BZ^r$-indexed filtration
$$
J(v)=\{f\in \Oo\ |\ \frakv(f)\succeq v\}.
$$
Note that the ideals $J(v)$ are defined for negative values of $v$ as well.
The filtration is decreasing:  if $u\preceq v$ then $J(u)\supset J(v)$.

\begin{definition}\label{def:Hil}
 The {\it Hilbert series of the multi-index filtration $J$} is
\begin{equation}
H(t_1,\ldots,t_r)=\sum_{v}\ h(v)\cdot t_1^{v_1}\cdots\, t_r^{v_r}\in
{\mathbb Z}[[t_1,t_1^{-1},\ldots, t_r,t_r^{-1}]],
\end{equation}
\end{definition}
\noindent where $h(v)=\dim_{\BC} \Oo/J(v)$.  Note that
\begin{equation}\label{eq:MAXh}
h(v)=h(\max\{v,0\}).
\end{equation}
 Hence $H$ is determined completely by
$H(t)|_{0\preceq v}:=\sum_{0\preceq v} h(v)t^v$.

\subsection{The Poincar\'e series} If $r=1$, then the Poincar\'e series
of the graded ring $\oplus_{v} J(v)/J(v+e_1)$ is   $P(t)=-H(t)(1-t^{-1})$. For general $r$, one
 defines the Poincar\'e series  similarly
\begin{equation}\label{eq:HPC}
P(t_1,\ldots,t_r)=-H(t_1,\ldots,t_r)\cdot \prod_i(1-t_i^{-1}).
\end{equation}
This means that the  coefficient $\pi_v$ of $
P=\sum_{v}\pi_v\cdot t_1^{v_1}\ldots t_r^{v_r}$ satisfies
\begin{equation}
\label{htopi}
\pi_{v}=\sum_{K\subset K_0}(-1)^{|K|-1}h(v+e_{K}).
\end{equation}
The space ${\mathbb Z}[[t_1,t_1^{-1},\ldots, t_r,t_r^{-1}]]$
is a module over the ring of Laurent power series, hence
the multiplication in \eqref{eq:HPC} is a well-defined. 
 One can check (using e.g. \eqref{eq:MAXh}) that the right hand side of \eqref{eq:HPC}
is a power series involving only nonnegative powers of $t_i$.

\subsection{Poincar\'e series and the Alexander polynomial}\label{ss:PA}
The topological aspect and importance of the Poincar\'e series is shown by the following theorem.
\begin{theorem}[\cite{cdg,cdg2}]
\label{Poincare vs Alexander}
Let $\Delta(t_1,\ldots,t_r)$ be the multi-variable Alexander polynomial of the link of $C$.
If \,$r=1$ then
 $P(t)(1-t)=\Delta(t)$,
while  $
P(t_1,\ldots,t_r)=\Delta(t_1,\ldots,t_r)
$  if \,$r>1$.
\end{theorem}

The Alexander polynomial is symmetric in the following sense.
For any $i\in K_0$ let
$\mu_{i}$ and $\delta_i$ (respectively $\mu(C)$ and $\delta(C)$)
be the Milnor number and the delta invariant  of $C_{i}$ (respectively of $C$), see \cite{book,Milnor}.
Let $(C_j,C_i)$ be the intersection multiplicities at 0 $(j\not=i)$.
Then, cf. \cite{Milnor}, $\mu_i=2\delta_i$, and $\mu(C)+r-1=2\delta(C)$.
Define $l=(l_1,\ldots,l_r)$ by  $$l_{i}=\mu_{i}+\sum_{j\neq i}(C_{j},C_{i}) \ \ \ \ \mbox{($1\leq i\leq r$)}.$$
Then $\Delta(t^{-1})=t^{-\mu(C)}\Delta(t)$ for $r=1$, and (e.g. by \cite{EN})
$$\Delta(t_1^{-1},\ldots,t_{r}^{-1})=
\left(\prod t_{i}^{1-l_{i}}\right)\cdot \Delta(t_1,\ldots,t_r) \ \ \  \ \mbox{for \,$r>1$}.$$
By
\cite{cdk,mozu},
the Hilbert function also satisfies similar symmetry properties
\begin{equation}
\label{sym}
h(l-v)-h(v)=\delta(C)-|v|,
\end{equation}
where $|v|=\sum_{i=1}^{r} v_i$. In particular, for
$v\succeq l$ one has
\begin{equation}\label{eq:NEWNEW}
h(v)=|v|-\delta(C).\end{equation}

\subsection{The equivalence of the Poincar\'e series and the Hilbert series}
For any subset $K=\{i_1,\ldots,i_{|K|}\}\subset K_0$, $K\not=\emptyset$,
 consider the curve $C_{K}=\cup_{i\in K}C_{i}$. As above, this germ defines
 the Hilbert series $H_{C_K}$ of $C_K$ in
 variables $\{t_i\}_{i\in K}$:
 $$H_{C_K}(t_{i_1},\ldots,t_{i_{|K|}})=\sum_{v}
  h^{K}(v)\cdot
 t_{i_1}^{v_{i_1}}\ldots t_{i_{|K|}}^{v_{i_{|K|}}}.$$
 By the very definition,
$H_{C_K}(t_{i_1},\ldots,t_{i_{|K|}})=H_C(t_1,\ldots,t_r)|_{t_i=0\ i\not\in K}$; or
\begin{equation}
\label{hrestriction}
\mbox{if  $v_i=0$ for all $i\notin K$, then} \ \ h^{K}(v)=h(v).
\end{equation}
Analogously, we also consider  the  Poincar\'e series of $C_K$:
$$P_{C_K}(t_{i_1},\ldots,t_{i_{|K|}})=\sum_{v}
 \pi^{K}_{v}\cdot
t_{i_1}^{v_{i_1}}\ldots t_{i_{|K|}}^{v_{i_{|K|}}}.$$
By definition, for $K=\emptyset$ we take $\pi^{\emptyset}_{v}=0.$

By \cite{Yamamoto} the multi-variable Alexander polynomial
(and hence by Theorem \ref{Poincare vs Alexander}
the Poincar\'e series $P(t)$) determines the embedded topological type of
$C$, in particular all the series  $\{P_{C_K}\}_{K\subset K_0}$. Nevertheless,  the
reduction  procedure from $P$ to $P_{C_K}$ is more complicated than
the analogs of (\ref{hrestriction}) valid for the Hilbert series.
Indeed, these formulae  are of type (see \cite{torres}):
 \begin{equation}\label{eq:redP}
P_{C_{K_0\setminus \{1\}}}(t_2,\ldots,t_r)=
P(t_1,\ldots,t_r)|_{t_1=1}\cdot \frac{1}{(1-t_2^{(C_1,C_2)})\cdots (1-t_r^{(C_1,C_r)})}.
\end{equation}

The next theorem inverts  (\ref{htopi}): we  recover $H$
from $P$.
The fact that $H$ can be recovered from $P$ was already proved in
 \cite[Corollary 4.3]{julioproj}. However, we wish to present a more general statement
 which also clarifies under what condition the inversion works, and which is applied for
certain coefficients provided by the Heegaard Floer link homology as well, cf.
Theorem \ref{thm: l space ss} and identity (\ref{Euler characteristic for g}).

\begin{theorem}\label{reconst}  Consider
$G(t_1,\ldots,t_r)=\sum_{v}t_1^{v_1}\ldots t_r^{v_r}\cdot g(v)
\in \BZ[[t_1,t_1^{-1},\ldots, t_r,t_r^{-1}]]$
with the following properties:

(a) $g(v)=g(\max\{v,0\})$;

(b)   $g(0)=0$.

(c) Fix $K\subset K_0$. We extend any  $v=(v_{i_1},\ldots, v_{i_{|K|}})$
to a vector with entries indexed by $K_0$ such that the entries indexed by $K_0\setminus K$ are zero.
(In this way $g(v)$ make sense.)
Then, we also require that the coefficients of $g$ satisfy (for any $K$) the following identities:
\begin{equation*}
\pi^K_{v}=\sum_{M\subset K}(-1)^{|M|-1}g(v+e_{M})
\ \ \ \ \ \mbox{for any  $v=(v_{i_1},\ldots, v_{i_{|K|}})$}.
\end{equation*}

Then $G$ is uniquely determined by $\{P_{C_K}\}_K$ (hence by $P$ too), and it satisfies
\begin{equation}
\label{hilbert}
G(t_1,\ldots,t_r)|_{0\preceq v}
=\frac{1}{ \prod_{i=1}^{r}(1-t_i)}\sum_{K\subset K_0}(-1)^{|K|-1}
\Big(\prod_{i\in K}t_i\Big)\cdot P_{C_K}(t_{i_1},\ldots,t_{i_{|K|}}).
\end{equation}
\end{theorem}
\begin{proof}
The identity (\ref{hilbert}) is equivalent to the following identity  of the coefficients:
\begin{equation}
\label{pitoh}
g(v)=\sum_{K\subset K_0}(-1)^{|K|-1}\sum_{0\preceq u\preceq v_K-e_{K}}\pi^{K}_{u}.
\end{equation}
We will prove the identity  (\ref{pitoh}) by a two-step induction:
the first induction is by the number of components $r$,
and the second one (for fixed $r$) is over the norm $|v|=\sum v_i.$

If $r=1$, then  (d) implies $\pi_{v}=g(v+1)-g(v).$
Hence $\sum_{0\le u\le v-1}\pi_{u}=g(v)$ since $g(0)=0$.

 Let us prove  (\ref{pitoh}) for the case when at least one of coordinates $v_i$ vanish.
We can assume that $v_r=0$. By \eqref{hrestriction} and the induction assumption  we get
$$
g(v)=g(v_1,\ldots,v_{r-1},0)=\sum_{K\subset \{1,\ldots, r-1\}}(-1)^{|K|-1}\sum_{0\preceq u\preceq v_{K}-e_{K}}\pi^{K}_{u}.
$$
On the other hand, in  (\ref{pitoh})
for all $K\subset K_0$ with $r\in K$ we get the vacuous restriction $0\le u_r\le -1$, hence
we get a nontrivial contribution only from terms with $K\subset \{1,\ldots, r-1\}$.

 Suppose now that $v$ has no vanishing coordinates and that we already proved (\ref{pitoh})
for  $v-e_{K}$ for all non-empty subsets $K\subset K_0$.
We can rewrite (d) as a linear equation on $\{g(v-e_K)\}_K$:
$$\pi_{v-e}=\sum_{K\subset K_0}(-1)^{r-|K|-1}g(v-e_{K}).$$
By the induction assumption for $K\neq \emptyset$ we have
$$g(v-e_{K})=\sum_{M\subset K_0}(-1)^{|M|-1}\sum_{0\preceq u\preceq
(v_{M}-e_{K\cap M}-e_{M})}\pi^{M}_{u},$$
and we should establish the same identity for $K=\emptyset$.
Therefore we need  to prove that
\begin{equation}
\label{subst}
\pi_{v-e}=\sum_{K\subset K_0}\sum_{M\subset K_0}(-1)^{r-|K|+|M|}
\sum_{0\preceq u\preceq (v_{M}-e_{K\cap M}-e_{M})}\pi^{M}_{u}.
\end{equation}
Let us fix $M$ and $u\preceq v-e$ and sum the expression $(-1)^{|K|}$
over all sets $K\subset K_0$  such that $u_i\le v_i-2$ for $i\in K\cap M$. This sum vanishes  unless
 $M=K_0$ and $u_i=v_i-1$ for all $i$, when it is 1.
This proves (\ref{subst}).
\end{proof}
\begin{corollary}
\label{cor:reconst}
(a) The Hilbert series satisfies the assumptions of
the above inversion theorem, hence $G=H$.

(b) The restricted Hilbert function  $H(t)|_{0\preceq v}$
of a multi-component curve is a rational function with denominator $\prod_{i=1}^{r}(1-t_i)^2$.
\end{corollary}
\begin{proof}
For (a) use identities \eqref{hrestriction} and \eqref{htopi} applied for $C_K$, while for (b)
Theorems \ref{reconst} and \ref{Poincare vs Alexander}. \end{proof}
\begin{remark}\label{remark:new}
Let us reprove the identity \eqref{eq:NEWNEW} using
(\ref{pitoh}). We analyze the different contributions.
For $K=\{i\}$ we have $\sum_{0\leq u_i\leq v_i-1}\pi^K_u=v_i-\delta(C_i)$.
For $K=\{i,j\}$ (since $P_{C_K}$ is a polynomial) we have
$\sum \pi^{\{i,j\}}_u=P_{C_K}(1,1)$. This equals $(C_i,C_j)$ by \eqref{eq:redP}.
By similar argument, for $|K|>2$ the contribution is zero.
Hence
$h(v)=\sum_i (v_i-\delta(C_i))-\sum_{i\not= j}(C_i, C_j)=|v|-\delta(C).$
\end{remark}

\subsection{The semigroup of $C$.} Important information about the  algebraic curve $C$
is coded in its  semigroup.
It is defined as $\Ss:=\{v\in \BZ^r\ |\ \mbox{there exists $f\in \Oo$
with $\frakv(f)=v$}\}$.
\begin{lemma}
\label{eq:semi}
The semigroup can be equivalently defined by the following condition:
$$
\Ss=\{v\in \BZ_{\geq 0}^r\ |\ h(v+e_i)>h(v) \ \ \mbox{for every \ $i=1,\ldots, r$}\}.
$$
Next, fix any $0\preceq v$ and $e_i$. Then  $h(v+e_i)=h(v)+1$ if there
is an element $u\in \Ss$  such that
$u_i=v_i$ and  $u_j\ge v_j$ for $j\neq i$. Otherwise $h(v+e_i)=h(v).$

In particular, $H$ and $\Ss$ determine each other.
\end{lemma}

\begin{proof}
If $h(v+e_i)>h(v)$ for all $i$, then there exist functions $f_i$
such that $\frakv_i(f_i)=v_i$ and $\frakv_j(f_i)\ge v_j$ for $j\neq i$.
 Therefore $\frakv(\sum_{i=1}^{r}\lambda_if_i)=v$
 for generic coefficients $\lambda_i$.
For the second part note that
$h(v+e_i)-h(v)=\dim J(v)/J(v+e_i)$.
This quotient space is trivial if  there is no function $f$ such that
$\frakv_i(f)=v_i$ and $\frakv_j(f)\ge v_j$ for $j\neq i$. Otherwise it is one-dimensional.
Indeed, if $\frakv_i(f_1)=\frakv_i(f_2)=v_i$
then there exists $\lambda\neq 0$ such that $\frakv_i(f_1-\lambda f_2)>v_i$.
If, moreover,  $\frakv_j(f_1),\frakv_j(f_2)\ge v_j$ for all $j\neq i$, then
$\frakv_j(f_1-\lambda f_2)\ge v_j$ too. Therefore $f_1-\lambda f_2\in J(v+e_i).$
\end{proof}
Next we establish the `matroid properties' of the function $h$.
\begin{lemma}
\label{lem:h matroid} (a)  Assume that  $h(v)=h(v+e_{i})$ for some fixed $i\in K_0$. Then
 $h(v+e_K)=h(v+e_K+e_{i})$
for any $K$ with $K\not\ni i$.

(b)
Suppose that $K_1,K_2\subset K_0$ and $v\in \BZ^r$. Then
$$h(v+e_{K_1})+h(v+e_{K_2})\ge h(v+e_{K_1\cap K_2})+h(v+e_{K_1\cup K_2}).$$

(c) For any base vector $e_i$ and $n\geq l_i$ one has
$h(v+(n+1)e_i)-h(v+ne_i)=1$.
\end{lemma}
\begin{proof} (a) Use $J(v+e_K+e_{i})=J(v+e_K)\cap J(v+e_{i})$.

(b) Replacing $v$ by $v+e_{K_1\cap K_2}$, we can assume that $K_1\cap K_2=\emptyset$. Therefore,
 $J(v+e_{K_1})\cap J(v+e_{K_2})=J(v+e_{K_1\cup K_2})$. Hence
$h(v+e_{K_1})+h(v+e_{K_2})-h(v)-h(v+e_{K_1\cup K_2})=
\dim J(v)/(J(v+e_{K_1})+J(v+e_{K_2}))\ge 0.$
For (c) use \eqref{eq:NEWNEW} and Lemma \ref{eq:semi}.
\end{proof}
\begin{remark}
It turns out (using e.g. \eqref{eq:NEWNEW} and
Lemma \ref{eq:semi}) that $l$ is the conductor of $\Ss$, in particular
$v\in\Ss$ whenever  $v\succeq l$.
\end{remark}

\begin{example}\label{ex:AN}
Consider the singularity $A_{2n-1}$ defined by the equation $x^2-y^{2n}=0$. Its Poincar\'e series equals $1+t_1t_2+\cdots+(t_1t_2)^{n-1}$,
and the Poincar\'e series of both its components equals $1/(1-t)$. The Hilbert series is given by the following equation:
$$
H(t_1,t_2)|_{0\preceq v}=\frac{1}{ (1-t_1)(1-t_2)}\left(\frac{t_1}{1-t_1}+\frac{t_2}{ 1-t_2}-t_1t_2(1+\ldots+(t_1t_2)^{n-1})\right).
$$
Therefore, for  non-negative integers $(v_1,v_2)$ one has
$$
h(v)=\begin{cases}
\max(v_1,v_2),& \mbox{\rm if } \min(v_1,v_2)<n,\\
v_1+v_2-n, & \mbox{\rm  otherwise.}\\
\end{cases}
$$
Figure \ref{a3hilb} illustrates this formula for the Hilbert function for $A_3$ singularity.
The points corresponding to the semigroup $\Ss$ are marked in bold.

\begin{figure}
\centering
\begin{tikzpicture}
\draw (0,0) node {\bf 0};
\draw (1,0) node {1};
\draw (2,0) node {2};
\draw (3,0) node {3};
\draw (4,0) node {4};
\draw (0,1) node {1};
\draw (1,1) node {\bf 1};
\draw (2,1) node {2};
\draw (3,1) node {3};
\draw (4,1) node {4};
\draw (0,2) node {2};
\draw (1,2) node {2};
\draw (2,2) node {\bf 2};
\draw (3,2) node {\bf 3};
\draw (4,2) node {\bf 4};
\draw (0,3) node {3};
\draw (1,3) node {3};
\draw (2,3) node {\bf 3};
\draw (3,3) node {\bf 4};
\draw (4,3) node {\bf 5};
\draw (0,4) node {4};
\draw (1,4) node {4};
\draw (2,4) node {\bf 4};
\draw (3,4) node {\bf 5};
\draw (4,4) node {\bf 6};

\draw [dashed] (0.2,0)--(0.9,0);
\draw [dashed] (1.2,0)--(1.9,0);
\draw [dashed] (2.2,0)--(2.9,0);
\draw [dashed] (3.2,0)--(3.9,0);
\draw [dashed] (0.2,1)--(0.9,1);
\draw [dashed] (1.2,1)--(1.9,1);
\draw [dashed] (2.2,1)--(2.9,1);
\draw [dashed] (3.2,1)--(3.9,1);
\draw [dashed] (0.2,2)--(0.9,2);
\draw [dashed] (1.2,2)--(1.9,2);
\draw [dashed] (2.2,2)--(2.9,2);
\draw [dashed] (3.2,2)--(3.9,2);
\draw [dashed] (0.2,3)--(0.9,3);
\draw [dashed] (1.2,3)--(1.9,3);
\draw [dashed] (2.2,3)--(2.9,3);
\draw [dashed] (3.2,3)--(3.9,3);
\draw [dashed] (0.2,4)--(0.9,4);
\draw [dashed] (1.2,4)--(1.9,4);
\draw [dashed] (2.2,4)--(2.9,4);
\draw [dashed] (3.2,4)--(3.9,4);
\draw [dashed] (0,0.3)--(0,0.8);
\draw [dashed] (1,0.3)--(1,0.8);
\draw [dashed] (2,0.3)--(2,0.8);
\draw [dashed] (3,0.3)--(3,0.8);
\draw [dashed] (0,1.3)--(0,1.8);
\draw [dashed] (1,1.3)--(1,1.8);
\draw [dashed] (2,1.3)--(2,1.8);
\draw [dashed] (3,1.3)--(3,1.8);
\draw [dashed] (0,2.3)--(0,2.8);
\draw [dashed] (1,2.3)--(1,2.8);
\draw [dashed] (2,2.3)--(2,2.8);
\draw [dashed] (3,2.3)--(3,2.8);
\draw [dashed] (0,3.3)--(0,3.8);
\draw [dashed] (1,3.3)--(1,3.8);
\draw [dashed] (2,3.3)--(2,3.8);
\draw [dashed] (3,3.3)--(3,3.8);
\draw [dashed] (4,0.3)--(4,0.8);
\draw [dashed] (4,1.3)--(4,1.8);
\draw [dashed] (4,2.3)--(4,2.8);
\draw [dashed] (4,3.3)--(4,3.8);

\end{tikzpicture}
\caption{Values of the Hilbert function for $A_3$}
\label{a3hilb}
\end{figure}

\end{example}

\begin{example}
Consider the singularity $D_5$ defined by the equation
$y\cdot (x^2-y^3) =0$.
Then
$$P(t_1,t_2)=1+t_1t_2^3, \ \ \ P_1(t_1)=\frac{1}{1-t_1}, \ \ \ P_2(t_2)=\frac{1-t_2+t_2^2}{1-t_2}.$$
One can check that $h(v_1,v_2)$ for non-negative $v_1$ and $v_2$ is  given by the following formula:
$$h(v_1,v_2)=
\begin{cases}
 v_1, & \mbox{\rm if    } v_2<3, v_1>0\\
 v_1+1, & \mbox{\rm if    } v_2=3, v_1>0\\
 v_2-1, & \mbox{\rm if   } v_1<2, v_2\ge 2\\
 v_1+v_2-3, & \mbox{\rm if  } v_1\ge 2, v_2\ge 4\\
 0,1,1, & \mbox{\rm if   } v_1=0 \quad \mbox{\rm and  } v_2=0,1,2.\\
\end{cases}$$
Figure \ref{d5hilb} illustrates the Hilbert function and the semigroup of $D_5$.
\end{example}

\begin{figure}
\centering
\begin{tikzpicture}
\draw (0,0) node {\bf 0};
\draw (1,0) node {1};
\draw (2,0) node {2};
\draw (3,0) node {3};
\draw (4,0) node {4};
\draw (5,0) node {5};
\draw (0,1) node {1};
\draw (1,1) node {1};
\draw (2,1) node {2};
\draw (3,1) node {3};
\draw (4,1) node {4};
\draw (5,1) node {5};
\draw (0,2) node {1};
\draw (1,2) node {\bf 1};
\draw (2,2) node {\bf 2};
\draw (3,2) node {\bf 3};
\draw (4,2) node {\bf 4};
\draw (5,2) node {\bf 5};
\draw (0,3) node {2};
\draw (1,3) node {\bf 2};
\draw (2,3) node {3};
\draw (3,3) node {4};
\draw (4,3) node {5};
\draw (5,3) node {6};
\draw (0,4) node {3};
\draw (1,4) node {3};
\draw (2,4) node {\bf 3};
\draw (3,4) node {\bf 4};
\draw (4,4) node {\bf 5};
\draw (5,4) node {\bf 6};
\draw (0,5) node {4};
\draw (1,5) node {4};
\draw (2,5) node {\bf 4};
\draw (3,5) node {\bf 5};
\draw (4,5) node {\bf 6};
\draw (5,5) node {\bf 7};

\draw [dashed] (0.2,0)--(0.9,0);
\draw [dashed] (1.2,0)--(1.9,0);
\draw [dashed] (2.2,0)--(2.9,0);
\draw [dashed] (3.2,0)--(3.9,0);
\draw [dashed] (4.2,0)--(4.9,0);
\draw [dashed] (0.2,1)--(0.9,1);
\draw [dashed] (1.2,1)--(1.9,1);
\draw [dashed] (2.2,1)--(2.9,1);
\draw [dashed] (3.2,1)--(3.9,1);
\draw [dashed] (4.2,1)--(4.9,1);
\draw [dashed] (0.2,2)--(0.9,2);
\draw [dashed] (1.2,2)--(1.9,2);
\draw [dashed] (2.2,2)--(2.9,2);
\draw [dashed] (3.2,2)--(3.9,2);
\draw [dashed] (4.2,2)--(4.9,2);
\draw [dashed] (0.2,3)--(0.9,3);
\draw [dashed] (1.2,3)--(1.9,3);
\draw [dashed] (2.2,3)--(2.9,3);
\draw [dashed] (3.2,3)--(3.9,3);
\draw [dashed] (4.2,3)--(4.9,3);
\draw [dashed] (0.2,4)--(0.9,4);
\draw [dashed] (1.2,4)--(1.9,4);
\draw [dashed] (2.2,4)--(2.9,4);
\draw [dashed] (3.2,4)--(3.9,4);
\draw [dashed] (4.2,4)--(4.9,4);
\draw [dashed] (0,0.3)--(0,0.8);
\draw [dashed] (0.2,5)--(0.9,5);
\draw [dashed] (1.2,5)--(1.9,5);
\draw [dashed] (2.2,5)--(2.9,5);
\draw [dashed] (3.2,5)--(3.9,5);
\draw [dashed] (4.2,5)--(4.9,5);

\draw [dashed] (1,0.3)--(1,0.8);
\draw [dashed] (2,0.3)--(2,0.8);
\draw [dashed] (3,0.3)--(3,0.8);
\draw [dashed] (4,0.3)--(4,0.8);
\draw [dashed] (5,0.3)--(5,0.8);
\draw [dashed] (0,1.3)--(0,1.8);
\draw [dashed] (1,1.3)--(1,1.8);
\draw [dashed] (2,1.3)--(2,1.8);
\draw [dashed] (3,1.3)--(3,1.8);
\draw [dashed] (4,1.3)--(4,1.8);
\draw [dashed] (5,1.3)--(5,1.8);
\draw [dashed] (0,2.3)--(0,2.8);
\draw [dashed] (1,2.3)--(1,2.8);
\draw [dashed] (2,2.3)--(2,2.8);
\draw [dashed] (3,2.3)--(3,2.8);
\draw [dashed] (4,2.3)--(4,2.8);
\draw [dashed] (5,2.3)--(5,2.8);
\draw [dashed] (0,3.3)--(0,3.8);
\draw [dashed] (1,3.3)--(1,3.8);
\draw [dashed] (2,3.3)--(2,3.8);
\draw [dashed] (3,3.3)--(3,3.8);
\draw [dashed] (4,3.3)--(4,3.8);
\draw [dashed] (5,3.3)--(5,3.8);
\draw [dashed] (4,0.3)--(4,0.8);
\draw [dashed] (4,1.3)--(4,1.8);
\draw [dashed] (4,2.3)--(4,2.8);
\draw [dashed] (4,3.3)--(4,3.8);
\draw [dashed] (4,3.3)--(4,3.8);
\draw [dashed] (0,4.3)--(0,4.8);
\draw [dashed] (1,4.3)--(1,4.8);
\draw [dashed] (2,4.3)--(2,4.8);
\draw [dashed] (3,4.3)--(3,4.8);
\draw [dashed] (4,4.3)--(4,4.8);
\draw [dashed] (5,4.3)--(5,4.8);

\end{tikzpicture}
\caption{Values of the Hilbert function for $D_5$}
\label{d5hilb}
\end{figure}

\subsection{The local hyperplane arrangements}\label{ss:ARR}  For any fixed $v$ let us consider the
set
\begin{equation*}
\Hh(v):=\{f\in \Oo\, :\, \frakv(f)=v\}=J(v)\setminus \bigcup_i\, J(v+e_i).
\end{equation*}
Since $J(v+e_i)$ is either $J(v)$ or one of its hyperplanes (cf. \ref{eq:semi}), $\Hh(v)$ is
either empty or it is a hyperplane
arrangement in  $J(v)$. This can be reduced to a finite dimensional
central hyperplane arrangement
$$\Hh'(v):=\frac{J(v)}{J(v+e_{K_0})}\setminus \bigcup_i \frac{J(v+e_i)}{J(v+e_{K_0})},$$
since $\Hh(v)\simeq J(v+e_{K_0})\times \Hh'(v)$. Note that both $\Hh'(v)$ and $\Hh(v)$
admit a free $\BC^*$--action (multiplication by nonzero scalar), hence one automatically has the
two projective arrangements $\BP\Hh'(v)=\Hh'(v)/\BC^*$ and $\BP\Hh(v)=\Hh(v)/\BC^*$ as well.
The following proposition can be deduced from \eqref{htopi} and inclusion-exclusion formula (see e.g. \cite{cdg,cdg2} and Lemma \ref{lem:L} below).

\begin{proposition}
The Euler characteristic of\ \  $\BP\Hh(v)$ (and of\ \ $\BP\Hh'(v)$) equals $\pi_{v}$, the coefficient of the Poincar\'e series $P(t)$ at $t^{v}$.
\end{proposition}

\subsection{Motivic Poincar\'e series}\label{ss:MP} The series $\Pp(t_1,\ldots, t_r;q)\in
\BZ[[t_1,\ldots, t_r]][q]$ is defined in \cite{cdg3} as a refinement of $P(t)$
as follows.   By definition, the coefficient  of
$t_1^{v_1}\ldots t_r^{v_r}$  is the (normalized) class of
$\BP\Hh'(v)$ in the Grothendieck ring of algebraic varieties. It turns out that the class
of a central hyperplane arrangement can always be expressed in terms of the class $\BL$
of the affine line. Indeed, one has:
\begin{lemma}\label{lem:L}
$V$ be a vector space and let $\Hh=\{\Hh_1,\ldots,\Hh_r\}$ be a collection of linear hyperplanes in $V$.
For a subset $K$ we define the rank function by
$\rho(K)=\codim \cup_{i\in K}\Hh_i$.
Then in the Grothendieck ring of varieties
(by the inclusion-exclusion formula) one has
 $$[V\setminus \cup_{i=1}^r\Hh_i]=\sum_{K\subset K_0}(-1)^{|K|}\ [\cap_{\alpha\in K}\Hh_\alpha]
 =\sum_{K\subset K_0}(-1)^{|K|}\ \BL^{\dim V-\rho(K)}.$$
 Since $[\BC^*]=\BL-1$, one also has $[(V\setminus \cup_{i=1}^r\Hh_i)/\BC^*]=[V\setminus \cup_{i=1}^r\Hh_i]
 /(\BL-1)$.
\end{lemma}
\begin{corollary}
The class of the (finite) local hyperplane arrangement $\Hh'(v)$ equals
$$[\Hh'(v)]=(\BL-1)[\BP\Hh'(v)]=\sum_{K\subset K_0}(-1)^{|K|}\ \BL^{h(v+e_{K_0})-h(v+e_K)}.$$
\end{corollary}
Replacing $\BL^{-1}$ by a new variable $q$, one can define (following \cite{cdg3}) the motivic Poincar\'e series
$\Pp(t;q)=\sum_{v}\pi_{v}(q)t^{v}$ by
$$ \pi_{v}(q):=\BL^{1-h(v+e_{K_0})}[\BP\Hh'(v)]\Big|_{\BL^{-1}=q}
=\frac{1}{1-q}\sum_{K\subset K_0}(-1)^{|K|}q^{h(v+e_K)}=$$
$$\sum_{K\subset K_0}(-1)^{|K|}\cdot\frac{q^{h(v+e_K)}-q^{h(v)}}{1-q}.
$$
Note that $\lim_{q\to 1}\Pp(t;q)=P(t)$. In \cite{intfun,mozu} several  properties of
 $\Pp(t_1,\ldots,t_r;q)$ are proved, e.g. it is a
rational function with denominator $\prod_{i=1}^{r}(1-t_{i}q)$. We will need the following.

\begin{lemma}\label{lem:supportMP}
 The support of $\Pp(t;q)$ is exactly $\Ss$. That is, $\pi_v(q)\not=0$ if and only if  $v\in \Ss$.
\end{lemma}

\begin{proof}
If $v\not\in\Ss$, then there exists $i\in K_0$ such that $h(v+e_i)=h(v)$ (cf. \ref{eq:semi}), hence
$\pi_v(q)=0$ by \ref{lem:h matroid}(a).
 If  $v\in S$, then $h(v+e_i)=h(v)+1$ for all $i$ and $h(v+e_{K})\ge h(v)+1$ for all subsets $K$,
 hence $\pi_{v}(q)=q^{h(v)}+$ higher order terms.  \end{proof}

By Theorem \ref{reconst}, $\Pp(t;q)$ and $P(t)$, in fact, determine each other.
\subsection{Conclusion} By the above discussions, the following objects associated with a
plane curve singularity carry the same amount of information: the multi-variable Alexander
polynomial $\Delta(t)$, the semigroup $\Ss$, the Hilbert series $H(t)$,
the Poincar\'e series $P(t)$ and the motivic Poincar\'e series $\Pp(t;q)$.
The role of the spaces $\Hh(v)$ will be crucial in the next parts:
we will compute their homology  using  the Orlik--Solomon
algebras of hyperplane arrangements.
This will connect  two other  objects:
the local lattice homology (associated with the weight function $h$)
and the Heegaard Floer link homology of the link of $C$.
This connection and
the `matroid properties' (\ref{lem:h matroid}) of the weight function $h$ are responsible
for the collapse of a spectral sequence connecting the Heegaard Floer link
homology with the lattice homology.

The Poincar\'e polynomials of all these cohomologies will
be identified with the coefficients of the motivic Poincar\'e series.

\begin{remark}\label{rem:Pdef}
In the above definition $\Oo_{\BC^2,0}$ can be replaced by $\Oo_C$. In this way, one can
extend all the above definitions of $H(t)$, $P(t)$, $\Pp(t;q)$, $\Ss$ to the case of
any (not necessarily plane) reduced curve singularity. The topological embedded--link invariant
$\Delta(t)$ has no analogue  in this general case. It is a nice challenge to find the analogue of the $HFL$--theory
(via $(H(t)$ as in this note) applied for a (non--planar) curve singularity.
\end{remark}

\section{Lattice homology}\label{s:lattice}

Lattice homology associated with the intersection lattice of a resolution of a normal surface
singularity was introduced in \cite{Nemethi08}, as a topological invariant of negative definite plumbed
 3--manifolds. For a possible generalization to algebraic knots, see the recent manuscript  \cite{OSZSt}.

In this section we introduce another
homology theory associated with curve singularities,
where the lattice and the corresponding weight function have a different nature.
In order to make a distinction  between the two cases we will call
the present theory {\it lattice homology of curve singularities via their normalizations}.
In fact, the definitions below extend identically to any,
not necessarily plane curve singularity, that is,  even if  $(C,0)$ does not have any
local {\it embedded link} in the 3--sphere.

\subsection{The general theory: lattice complex, filtrations, lattice homology.}\label{ss:3.1} \

 In this subsection  we present the general theory of lattice homology
 associated with an arbitrary weight function. This will be specialized for the function $h$
 in subsection \ref{ss:3.2}, and for the $HFL$--weight function $g$ given by Heegaard Floer link theory in
 section \ref{s:L} (see also section \ref{s:La}).

We will use the cubes in $\BR^r$ with vertices in the  lattice points  $\BZ^{r}$. Every such cube $\bsq(v,K)$,
where $v\in \BZ^r$ and $ K\subset K_0$, is defined as
$$\bsq=\bsq(v,K)=\{x\in \BR^{r}:v\preceq x\preceq v+e_{K}\},\ \dim \bsq(v,K)=|K|.$$

We consider $\bsq$ with its natural orientation (as a subset of $\BR^r$).
In the classical cubical homology, the chain complex is a free $\BZ$-module with generators $\sq=\sq(v,K)$ corresponding to the cubes $\bsq(v,K)$,
and  the differential can be written as
 $\partial(\sq)=\sum_{i}\varepsilon_{i}\sq_{i},$
 where $\bsq_i$ are oriented codimension 1 faces of the cube $\bsq$.
 \footnote{Note that here and below a full square  $\bsq$ denotes a geometric object ($|K|$--dimensional solid cube in $\BR^r$), while a hollow square $\sq$ denotes
 the corresponding  abstract generator in a chain complex.}

\begin{definition}
Let us choose a function $w:\BZ^r\to \BZ$, which will be called {\em weight function}.
We define the {\it weight  of a cube} by  $$w(\sq)=\max \{w(v): v\in \bsq\cap \BZ^r\}.$$
If $w(v)$ is non-decreasing (that is, $w(u)\leq w(v)$ whenever $u\preceq v$),
 then, in fact, $w(\sq(v,K))=w(v+e_{K})$.
\end{definition}

\begin{definition}
The {\it lattice complex}
$\Ll^{-}_{w}$ associated with a  weight function $w$ is a free $\BZ[U]$-module generated by all cubes
$\sq=\sq(v,K)$ with the following $\BZ[U]$--linear differential:
\begin{equation}\label{delta}
\partial_{U}(\sq)=\sum_{i}\varepsilon_{i}U^{w(\sq)-w(\sq_{i})}\sq_{i}.
\end{equation}
\end{definition}

One verifies that $\partial_{U}^2=0$.
We set $\deg U=-2$ and we introduce the
{\it homological grading} of a generator by
\begin{equation}\label{eq:HOMDEG1}
\deg (U^m\sq)=-2m+\dim (\sq)-2w(\sq).
\end{equation}
The differential $\partial_{U}$ decreases the homological grading by 1.

\begin{remark}
It is clear that the weight functions $w(v)$ and $w(v)+const$ define isomorphic lattice
complexes. However, the shift of $w$ by a constant
induces a shift in the  homological degree \eqref{eq:HOMDEG1} as well.
\end{remark}

\begin{definition}
 We define a $\BZ^r$-indexed filtration on the complex $\Ll^{-}_w$ as follows:  the
 subcomplex $\Ll^{-}_w(u)$ ($u\in \BZ^r$)
is generated over $\BZ[U]$ by all the cubes $\sq(v,K)$ with $v\succeq u$.
\end{definition}

It is easy to see that
 $\partial_{U}$ preserves the filtration,
so $\Ll^{-}_w(u)$ is a subcomplex of $\Ll^{-}_w$ for all $u$.
The next theorem shows that the homologies of different subcomplexes,  and the homology of
$\Ll^-_w$ itself, is simple (compatibly with facts from  Heegaard Floer link theory, cf.
Theorem \ref{th:lspace}).

\begin{theorem}
\label{homology of l(u)}
\ Assume that $w$ is non-decreasing. Then the following facts hold.

(a) The homology of $\Ll^{-}_w(u)$ is isomorphic to $\BZ[U]$ (as $\BZ[U]$--module).
It is generated by the class $\Box(u,\emptyset)$ of  homological degree $-2w(u)$.

(b) If additionally
$w(v)=w(\max\{0,v\})$, then
the inclusion $\Ll^{-}_w(0)\subset \Ll^{-}_w$ induces an isomorphism at the level of homology.
In particular, the homology of $\Ll^-_w$ is $\BZ[U]$.
\end{theorem}
Note that the assumptions on $w$ are satisfied by the Hilbert function $h$
of a curve, see \eqref{eq:MAXh}.

\begin{proof}
(a) For every $k\ge w(u)$ let us define the topological space
$S_{k}(u):=\bigcup\ \bsq(v,K)\subset \BR^r$, where the union is over
cubes\, $\bsq(v,K)$ with $v\succeq u$ and $w(\Box(v,K))=w(v+e_K)\leq k$.
Note that $\bsq(u,\emptyset)$ satisfies the requirements, hence $S_k(u)$ is non-empty, it contains $u$.

Similarly to  \cite[Theorem 3.1.12]{Nemethi08}, we show the following isomorphism of
$\BZ$--modules for any $q\in \BZ$:
\begin{equation}\label{eq:S}
H_{q}(\Ll^{-}_w(u))=\bigoplus_{\substack{k\geq w(u)\\ q'-2k=q}}H_{q'}(S_{k}(u),{\BZ}).
\end{equation}
This can be proved as follows. Let $\Ccc_*(S_k(u))$ be the usual cubical chain complex
of $S_k(u)$.  $\oplus_{k\geq h(u)}\Ccc_*(S_k(u))$ is their direct sum (as chain complexes), where we prefer
 to write $(k,\alpha)$ for an element of the $k$-th component.
 We define the $\BZ$--linear morphism $\Phi:\Ll^-_w(u)\to
\oplus_{k\geq h(u)}\Ccc_*(S_k(u))$ by $U^l\Box(v,K)\mapsto (l+w(v+e_K),\Box(v,K))$, where
the latter cube $\Box(v,K)$ is considered in $\Ccc_{|K|}(S_k(u))$, positioned in the component  $k=l+w(v+e_K)$.
This is a linear isomorphism with inverse $(k, \Box(v,K))\mapsto U^{k-w(\Box(v,K))}\Box(v,K)$.
Moreover,
$\Phi(\partial_U(U^l\Box(v,K)))=\partial \Phi(U^l\Box(v,K))$
(where $\partial$ means the direct sum of usual boundary operators of $\Ccc_*(S_k(u))$).

Furthermore, multiplication by $U$ in $\Ll^-_w(u)$ corresponds to the operator
$(k,\Box)\mapsto (k+1,i(\Box))$, where $i$ is induced by the
inclusion $S_k\hookrightarrow S_{k+1}$ at the level of $\oplus_{k\geq w(u)}\Ccc_*(S_k(u))$.

Hence, $\Phi$ induces a morphism at the level of homology. If the homological degree $-2l+|K|-2w(v+e_K)$
of $U^l\Box(v+e_K)$ is denoted by $q$, then its homological class is sent by $\Phi_*$ into
$H_{q'}(S_k)$, where $q'=|K|$ and $2k=2(l+w(v+e_K))=|K|-q=q'-q$. Hence (\ref{eq:S}) follows.

Next,  we prove that $S_{k}(u)$ is contractible for all $k$.
Indeed, since $w$ is non--decreasing,
if $\bsq(v,K)\subset S_k(u)$, then the set $S_{k}(u)$ contains the whole parallelepiped
$\{x:u\preceq x\preceq v+e_K\}$. Such a space can be contracted to the lattice point  $u$.

In particular, in (\ref{eq:S}) $q'$ should be zero, $q=-2k$ and $k\geq w(u)$, while
$H_0(S_k(u))= \BZ$. This means that $H_q(\Ll^-_w(u))$ is zero unless
$q=-2w(v)-2l$ for $l\geq 0$, and in this case it is $\BZ$ corresponding to
the generator $\Box(u,\emptyset)$ considered in $S_{w(u)+l}$; or, in the homology of
$\Ll^-_w(u)$, to the class of $U^l\Box(u,\emptyset)$. Hence
$$H_{*}(\Ll^{-}_w(u))=\BZ[U]\cdot \sq(u,\emptyset).$$

(b) For $1\leq p\leq r$ we define the sub-complex $\Ll_{w,p}^-$ of $\Ll^-_w$ generated over $\BZ[U]$
by cubes $\Box(v,K)$ with $v=(v_1,\ldots,v_r)$, $v_i\geq 0$ for $1\leq i\leq p$. Then
$\Ll^-_{w,r}=\Ll^-_w(0)$ and we also  set $\Ll^-_{w,0}:=\Ll^-_w$.   We show  that $\Ll^-_{w,p}\subset
\Ll^-_{w,p-1}$ is a homotopy equivalence, hence (b) follows by induction on $p$.

Let $(Q_{p-1},\partial^Q)$ be the quotient complex $\Ll^-_{w,p-1}/\Ll^-_{w,p}$.
It is generated by cubes $\Box(v,K)$ with $v_i\geq 0$ for $1\leq i\leq p-1$ and $v_p<0$.
Note that for such a lattice point one has $w(v)=w(v+e_p)$. Therefore, $(Q_{p-1},\partial^Q)$
is a tensor product of two complexes $(R_{p-1},\partial^R)\otimes (T,\partial^T)$, where
$(T,\partial^T)$ is the quotient lattice complex $\BR/\BR_{\geq 0}$ associated with the constant
zero weight (this corresponds to the $p$-th coordinate). More precisely, $T$ is generated by
0--cubes $a_n:=\Box(n,\emptyset)$  and 1--cubes $\alpha_n:=\Box(n,\{1\})$ for  $n\in \BZ_{<0}$, and $\partial^T(\alpha_n)=
a_{n+1}-a_n$ (with the notation $a_0=0$).
It is easy to check that the homology of $(T,\partial^T)$ is trivial,
hence $H_*(Q_{p-1},\partial^Q)=0$ too.
\end{proof}

The point is that the really
 interesting information is codified in the associated graded versions and in the pages of the
corresponding spectral sequences converging to $H_*(\Ll^-_w)$.

\begin{definition}
We define the multi-graded direct sum  complex $\gr \Ll^{-}=\oplus_v \gr_v\Ll^{-}$, where
$$\gr_v \Ll^{-}=
\Ll^{-}(v)/\textstyle{\sum_{i=1}^{r}}\Ll^{-}(v+e_i)
$$
with induced boundary operator $\gr\partial_U$.
The graded homology group $\HL^{-}=\oplus_v \HL^{-}(v)$, where
$$\HL^{-}(v):=H_{*}(\gr_v \Ll^{-}, \gr_v\partial_U),$$
is called the {\it local lattice homology associated with the weight function $w$}.
It has an induced $\BZ[U]$ module structure.
\end{definition}

\begin{remark}\label{rem:SpecSeq}
Consider the filtration $\{F_n\}_{n\in \BZ}$, where the sub-complex $F_n$ of $\Ll^-_w$
is generated over $\BZ[U]$ by cubes $\Box(v,K)$ with $|v|\geq n$. Then
$F_n/F_{n+1}=\oplus _{|v|=n}\gr_v\Ll^-_w$, and $\oplus_n F_n/F_{n+1}=\gr\Ll^-_w$.
Therefore, there exists a spectral sequence
$$E^1=H_*(\gr \Ll^-_w) \ \Rightarrow \ E^\infty =H_*(\Ll^-_w).$$
\end{remark}

\begin{remark}\label{rem:bigr} {\bf The bigrading of $\Ll^-_w$.}
 \
The following bigrading helps to enlighten some hidden structure of the lattice homology (cf. part (3) of Theorem \ref{zlat}
and the proof after it).
We define the following improvement of the homological grading (\ref{eq:HOMDEG1})
\begin{equation*}
 \bdeg (U^m\sq)=(-2m-2w(\sq),\dim (\sq))\in\BZ^2.
\end{equation*}
Then the boundary operator $\partial_U$ has bidegree $(0,-1)$. In particular, $HL^-(v)$ is also
bigraded. Let $HL^-_{a,b}(v)$ denote the corresponding $(a,b)$--component of $HL^-(v)$.
\end {remark}

\subsection{The case of algebraic curves}\label{ss:3.2}
Given a curve singularity $C$ with Hilbert function $h(v)$, one can consider
the lattice complex with the weight function $v\mapsto h(v)$ (which is non-decreasing).
In this case we will abbreviate the notation to
$\Ll^{-}=\Ll^{-}_{C}:=\Ll^{-}_{h}.$

\begin{theorem}
\label{zlat}
(1)  Consider the motivic Poincar\'e series of $C$,
$\Pp(t;q)=\sum_{v}\pi_{v}(q)t^{v}$.
Then the Poincar\'e polynomial of\, $\HL^-(v)$,
namely $P^{\Ll^-}_{v}(t):=\sum_{i} t^{i}\rank H_{i}(\gr_v \Ll^{-},\gr_v\partial_U)$,
satisfies
\begin{equation}
\label{hfminus}
P^{\Ll^-}_{v}(-t^{-1})=t^{h(v)}\cdot \pi_{v}(t).
\end{equation}

In particular, $(-1)^{h(v)}\cdot \pi_v(-q)$ is a polynomial in $q$ with non-negative coefficients.

Moreover, the Euler characteristic $P^{\Ll^-}_{v}(-1)=\sum_{i} (-1)^i\rank H_{i}(\gr_v \Ll^{-})$ equals
$ \pi_{v}(1)=\pi_v$, the $v$--coefficient of the Poincar\'e series.

(2)  Furthermore,
$H_{-2h(v)-p}(\gr_v \Ll^{-},\gr_v\partial_U)\simeq H^{p}(\BP\Hh'(v),\BZ)$, where  $\BP\Hh'(v)$
is the complement of the projective hyperplane arrangement defined in \ref{ss:ARR}.

(3) If $H_{a,b}(\gr_v \Ll^{-},\gr_v\partial_U)\not=0$ then necessarily $a+2b=-2h(v)$ (or, $\deg=-2h(v)-b$).

(4) The $U$--action on $H_{*}(\gr_v \Ll^{-},\gr_v\partial_U)$ is trivial.
\end{theorem}

We postpone the proof  of Theorem \ref{zlat} till subsection \ref{sec:pf}, where we will use
hyperplane arrangements and their Orlik-Solomon algebras. The surprising similarities between the Orlik--Solomon complex
and the lattice complex will be used deeply.
Nevertheless, here we will show how (\ref{hfminus}) can be deduced
 from (3). This also shows that (\ref{hfminus})
is not the output of a merely homological manipulation, but it reflects a deeper vanishing property
of the Orlik-Solomon algebras.

\begin{proof} (3)$\Rightarrow$ (\ref{hfminus}). \
For an bigraded $\BZ$--module $\{H_{a,b}\}_{a,b}$ set the virtual  Poincar\'e polynomial
$P_{\bdeg}^{vir}(t):=\sum_{a,b}(-1)^bt^a\rank H_{a,b}.$
In particular, this applied to $\gr_v\Ll^-$,
and counting the bi-degrees of the cubes
$\{U^m\sq(v,K)\}_{m\geq 0,\, K\subset K_0}$, we get
$$P_{\bdeg}^{vir}(t)(\gr_v\Ll^-)=\sum_{K\subset K_0}(-1)^{|K|}\cdot \frac{t^{-2h(v+e_K)}}{1-t^{-2}}=\pi_v(t^{-2}).$$
Since the differential $\partial_U$ has bi-degree $(0,-1)$, the virtual Poincar\'e polynomials of the complex and its homology coincide and we get
$P_{\bdeg}^{vir}(t)(HF^-(v))=P_{\bdeg}^{vir}(t)(\gr_v\Ll^-)$, hence
\begin{equation}\label{eq:bdegp}
 \sum_{a, K}(-1)^{|K|} t^a\cdot \rank\, HL^-_{a,|K|}(v)=\pi_v(t^{-2}).
\end{equation}
Then (\ref{hfminus}) is equivalent to
$$
\sum_{a,K}(-1)^{|K|}t^a\cdot \rank \, HL^-_{a,|K|}(v)=
\sum_{a,K}(-1)^{|K|}t^{2a+2h(v)+2|K|}\cdot \rank \, HL^-_{a,|K|}(v).
$$
But this is true, since $a=2a+2h(v)+2|K|$ whenever $HL^-_{a,|K|}(v)\not=0$ by (3).
\end{proof}

\begin{corollary}
\label{support hl}
$v\in \Ss$ if and only if $\HL^-(v)\neq 0$. For any $v\in\Ss$ one has
$P^{\Ll^-}_{v}(-t^{-1})=t^{2h(v)}+$ higher order terms.
(This shows that the class of  $\sq(v,\emptyset)$ does not vanish in $\HL^-(v)$.) In particular,
$P^{\Ll^-}_{v}(t)$ and $\pi_v(q)$ determine each other.
\end{corollary}

\begin{proof} Use Lemma \ref{lem:supportMP} (and its proof) and the identity (\ref{hfminus}).
\end{proof}

\subsection{Example. The case of a curve with one component}\

Suppose that $r=1$. We will abbreviate
$\sq(v,\emptyset) =a_{v}$, $\sq(v,\{1\})=\alpha_{v}$.
If $v\not\in\Ss$ then  $(\gr_v\partial_{U})(\alpha_{v})=a_{v}$,
 hence $\HL^{-}(v)=0$.
If $v\in\Ss$ then  $(\gr_v\partial_{U})(\alpha_{v})=Ua_{v}$,
 hence $\HL^{-}(v)=\BZ\langle a_{v}\rangle$ of homological degree $-2h(v)$.
Hence for $v\in\Ss$ one has  $P^{\Ll^-}_{v}(t)=t^{-2h(v)}$ compatibly with  $\Pp(t;q)=\sum_{v\in\Ss} q^{h(v)}t^v$.

Furthermore, the spectral sequence from Remark \ref{rem:SpecSeq} satisfies
 $E^1\simeq E^\infty\simeq \BZ[U]$ as $\BZ$-modules. Nevertheless, $E^1\not \simeq E^\infty$
 as $\BZ[U]$ modules: $E^1$ has trivial  $U$--action, while   in $E^\infty$
the $U$--action sends the generator of a semigroup element to the generator of the consecutive semigroup
element.

\begin{remark} ({\it The $U=0$ (or ``hat''--) version.})
\
(a) It is interesting to consider the complex $\Ll^-_{U=0}$ too (obtained from $\Ll^-$ via substitution
$U=0$),   generated over $\BZ$ by the cubes and
boundary operator given by (\ref{delta}) with substitution $U=0$.
Then $H_*(\Ll^-_{U=0})=\BZ$ (generated by the class of  $a_0$).
Moreover, the filtration $F_n':=F_n|_{U=0}$ induces a spectral sequence $\{E^k_{U=0}\}_k$.
$F_n' /F_{n+1}'$
is generated over $\BZ$ by all $a_v $ and $\alpha_v$, and the only non-trivial components of
the boundary map are the isomorphisms
$\BZ\langle \alpha_v\rangle \to \BZ\langle a_v\rangle$ for any $v\not\in \Ss$.
Hence  $E^1_{U=0}$ is $\oplus_{v\in \Ss}\BZ\langle a_v, \alpha_v\rangle$
of homological degrees $-2h(v)$ and $-2h(v)-1$ repectively.   The
non-trivial components of the $d^1:E^1_{U=0}\to E^1_{U=0}$ operator are the isomorphisms
$\BZ\langle \alpha_v \rangle \to \BZ\langle a_{v+1} \rangle$ whenever
both $v$ and $v+1$ are elements of $\Ss$. Hence, the $E^2_{U=0}$ term is
  \begin{equation*}
E^2_{U=0}(v)=\left\{ \begin{array}{ll}
\BZ\langle a_v,\alpha_v\rangle & \mbox{if $v\in \Ss$, \ $v-1\not\in\Ss$, \ $v+1\not\in \Ss$},\\
\BZ\langle a_v\rangle & \mbox{if $v\in \Ss$, \
$v-1\not\in \Ss$, \ $v+1\in \Ss$},\\
\BZ\langle \alpha_v\rangle & \mbox{if $v\in \Ss$, \ $v-1\in\Ss$, \
$v+1\not\in \Ss$},\\
0 & \mbox{otherwise}.
\end{array} \right.
\end{equation*}
The parity of the homological degree provides a $\BZ_2$ grading
$\{E^2_{U=0}\}_{\epsilon} $ of $E^2_{U=0}$,
where $\epsilon\in\{0,1\}$ has the same parity as the homological degree.
Then, since $\Delta(t)=(1-t)\sum_{v\in \Ss}t^v$,
$$\sum_{v,\epsilon} (-1)^\epsilon \ \rank( E^2_{U=0}(v)_{\epsilon})\,t^{v+\epsilon}\ = \ \Delta(t).$$
Since for irreducible plane curves $\Ss$ and $\Delta$ classifies the topological type of the
knot of $C$, cf. \cite{Yamamoto}, both  $E^1_{U=0}$ and  $E^2_{U=0}$ terms contain the complete information
about the local  topological type of $C$.
 Note also that
$E^2_{U=0}$ is supported in $[0,\mu]$, where $\mu=2\delta$ is the Milnor number of $C$, and
 $v\mapsto \mu-v-2\epsilon$ is a symmetry of $E^2_{U=0,\epsilon}$ which preserves the $\epsilon$--degree.

The $E^\infty_{U=0}$ term is $H_*(\Ll^-_{U=0})=\BZ$.

(b)
The short exact sequence of complexes $0\to \Ll^-\stackrel{U}{\longrightarrow} \Ll^-\to
\Ll^-_{U=0}\to 0$ induces a long exact sequence connecting the groups $HL^-(v)$ with
 the `$U=0$'--counterparts, whose explicit description is left to the reader.

\end{remark}

\subsection{Example. The case of a curve with two components}\

We will abbreviate
$\sq(v,\emptyset) =a_{v}$, $\sq(v,\{1\})=\alpha_{v}$
$\sq(v,\{2\})=\beta_{v}$ and  $\sq(v,\{1,2\})=\Gamma_{v}$.
By the general theory, if $v\not\in \Ss$ then  $HL^-(v)=0$.
If $v\in\Ss$ there are two cases.

a) $h(v)=
h(v+e_1+e_2)-1$;
$\alpha_v\mapsto Ua_v$, $\beta_v\mapsto Ua_v$, $\Gamma_v\mapsto \alpha_v-\beta_v$,
then  $\HL^-(v)=\BZ\langle a_v\rangle$ of homological degree $-2h(v)$. (In this case, $\BP\Hh(v)=$point.)

b) $h(v)=
h(v+e_1+e_2)-2$;
$\alpha_v\mapsto Ua_v$, $\beta_v\mapsto Ua_v$, $\Gamma_v\mapsto U\alpha_v-U\beta_v$,
then  $\HL^-(v)=\BZ\langle a_v, \alpha_v-\beta_v\rangle$ of homological degrees $-2h(v),-1-2h(v)$.
(Cf. with  $\BP\Hh(v)=\BP^1\setminus2$ points.)

In case (b) the Euler characteristic of $\HL^-(v)$ (and the corresponding coefficient in the Alexander polynomial) vanishes,
but the homology and the coefficient in the motivic Poincar\'e series do not vanish. This case appears, for example,
for all $v$ in the conductor of $C$.


Using   Figures \ref{a3hilb} and \ref{d5hilb}, one can compute the  $\HL^-$ for the singularities of
types $A_3$ and $D_5$. The analogous computation for the two-component singularity $A_{2n-1}$ agrees with the computations
of the Heegaard Floer link  homology in \cite{os3}.
An explicit  computation in the case $A_1$ is given in section \ref{NewExample}.

\subsection{Application to the theory of deformations of singularities}\

In this subsection we consider deformations of plane curve singularities.
>From topological point of view,
they induce cobordisms between the corresponding links in the three-sphere,
hence maps between their Heegaard Floer link homologies. We present here the analogous maps
in lattice homology, under the restriction that
the central fiber of the deformation is irreducible (while
the generic fiber is allowed to have several components).

We wish to emphasize that semicontinuity results for different singularity invariants are crucial
in the deformation theory of singularities, since they might provide more information about the
(open) problem of adjacencies of singularity types.

\begin{proposition}
\label{deform}

 Let $(C',0)$ be a curve singularity with $r$ irreducible components, and assume that it is
a deformation of an irreducible germ  $(C,0)$.
Then $h_{C'}(v)\ge h_{C}(|v|)$ for every  $v\in \BZ^r$.
\end{proposition}

\begin{proof}\footnote{We thank Maria Pe Pereira and Patrick Popescu-Pampu for noting a gap in the first version of the proof of this proposition.}
By Corollary \ref{cor:reconst}, the Hilbert function is determined by the topological type of a singularity. Consider the family of curves $C_t$ with the central fiber $C_0=C$ and the generic fiber $C_t$ topologically equivalent to $C'$. Let us fix $v\in \BZ^r$. One can assume that $h_{C_t}(v)$ is constant for small enough (but nonzero) $t$.

We get a family of subspaces $J_{C_t}(v)$ in $\Oo$ (or rather 
in a sufficiently large jet space $j_N\Oo$) of fixed codimension $h_{C_t}(v)=h_{C'}(v)$. Since the Grassmannian $\Gr(h_{C'}(v),j_N\Oo)$ is compact, this family has a well defined limit $J_{0}(v)=\lim_{t\to 0} J_{C_t}(v)$.

 Let us prove the inclusion $J_0(v)\subset J_{C}(|v|)$. Indeed, every function $g$ in this limiting subspace is a limit of a sequence of functions $g_t$ intersecting $C_t$ with multiplicity at least $|v|$, so by the semicontinuinty of the intersection multiplicity $g$ should intersect $C$ with multiplicity at least $|v|$ too. Therefore  $J_0(v)\subset J_{C}(|v|)$ and $h_{C}(|v|)=\codim J_{C}(|v|)\le \codim J_{0}(v)=h_{C'}(v).$
\end{proof}

After the first version of this paper appeared on arXiv, Borodzik and Livingston \cite{BoLi} gave an alternative proof of this proposition
(only for $\delta$--constant deformations)
using Heegaard Floer theory.

\begin{theorem}
Suppose that a (possibly reducible) curve $C'$ is a deformation of an irreducible curve $C$.
Then there exists a natural chain map $\phi: \Ll^{-}_{C'}\to \Ll^{-}_{C}$, with
$\phi(\Box(v,\emptyset))=U^{h_{C'}(v)-h_C(|v|)}\Box(|v|,\emptyset)$ and
$\phi(\Ll^{-}_{C'}(v))\subset \Ll^{-}_{C}(|v|)$   for any $v\in \BZ^r$. Moreover, for any $v$,
the induced map
 $$\phi_{*}(v):H_{*}(\Ll^{-}_{C'}(v))\to  H_{*}(\Ll^{-}_{C}(|v|))$$
 is the multiplication by $U^{h_{C'}(v)-h_C(|v|)}:\BZ[U]\langle \Box(v,\emptyset)\rangle\to
 \BZ[U]\langle \Box(|v|,\emptyset)\rangle$, hence it is injective.
\end{theorem}

\begin{proof}
Let us define a map $\phi$ acting on the generators of the lattice complex as follows.
For a $0-$ or a $1-$ dimensional cube in $\BZ^r$ one can define its natural projection onto $\BZ$ by
$$p(\Box(v,\emptyset)):=\Box(|v|,\emptyset); \ \ \  p(\Box(v,\{i\})):=\Box(|v|,1).$$
Then for an arbitrary cube $\Box$ define
$$\phi(\Box):=\begin{cases}
U^{h_{C'}(\Box)-h_{C}(p(\Box))}p(\Box) \ \ & \mbox{if}\ \dim \Box\le 1,\\
0&  \mbox{if}\ \dim \Box> 1.\\
\end{cases}
$$
By Lemma \ref{deform} the power of $U$ above is nonnegative, hence $\phi$ is well-defined.
It preserves the filtration on $\Ll^{-}$ and it
commutes with the differentials (by a straightforward computation left to the reader).
The injectivity of $\phi_{*}(v)$  follows from Theorem \ref{homology of l(u)}.
\end{proof}

We plan to study deformation theoretical applications
in more details in the future.

\section{Central hyperplane arrangements}\label{s:Arran}

\subsection{Matroids and rank functions}
\label{ss:OS}

\begin{definition}\label{def:St}  (a) (\cite{stanley})
Let $K_0$ be a finite set. A function $\rho$, assigning a non-negative integer to any subset $K\subset K_0$,
is called a {\em rank function},
if
\begin{enumerate}
\item $0\le \rho(K)\le |K|$.
\item If $K_1\subset K_2$ then $\rho(K_1)\le \rho(K_2)$.
\item For every pair of subsets $K_1$ and $K_2$ one has
$$\rho(K_1\cap K_2)+\rho(K_1\cup K_2)\le \rho(K_1)+\rho(K_2).$$
\end{enumerate}

(b) A {\em matroid} $M=(K_0,\rho)$ is a finite set $K_0$ with a rank function $\rho$ defined on it.

(c) The characteristic polynomial of a matroid $M=(K_0,\rho)$ is defined as
$$\chi_{M}(t)=\sum_{K\subset K_0}(-1)^{|K|}t^{\rho(K_0)-\rho(K)}.$$
\end{definition}

\begin{remark}
Some authors define the characteristic polynomial using the M\"obius function
of a matroid. This definition is equivalent to the present one, see  e.g.  \cite[Theorem 2.4]{stanley}.
\end{remark}

Let $h(v)$ denote the Hilbert function of a plane curve singularity.
Let us fix $K_0=\{1,\ldots,r\}$ and for every $v$ consider the following
function on subsets of $K_0$:
$$\rho_{v}(K):=h(v+e_{K})-h(v)=\dim J(v)/J(v+e_{K}).$$
Then  Lemmas \ref{eq:semi} and \ref{lem:h matroid} show that
for every $v$ the function $\rho_{v}$ is a rank function on $K_0$.

We will call $\rho_{v}$ the rank function for a the {\it local matroid $M_{v}$}.
In the space $J(v)$ we have $r$ subspaces $J(v+e_i)$ of codimension 0 or 1.
If $v\in\Ss$, then the set of functions with valuation $v$ can be represented
as a complement of a  hyperplane arrangement (cf. \cite{mozu}, or \ref{ss:ARR} here).
If $v\not\in \Ss$, then $J(v)=J(v+e_{i})$ for some $i$ (cf. Lemma \ref{eq:semi}), hence
 $J(v+e_K)=J(v+e_K+e_{i})$
for any $K$ with $K\not\ni i$ by \ref{lem:h matroid}. Therefore, in this case, by pairwise
 cancelation, $\chi_{M_v}(t)=0$.

\subsection{Some general facts on central hyperplane arrangements}\label{ss:4.2} \
Let $V$ be a
vector space and let $\Hh=\{\Hh_1,\ldots,\Hh_r\}$ be a collection of linear hyperplanes in $V$.
For a subset $K$ of $K_0=\{1,\ldots,r\}$ we define
$\rho(K)=\codim \cap_{i\in K}\,\Hh_i$.
One can check that $\rho$ is a rank function on $K_0$.
Let us denote by $\chi_{\Hh}(t)$ its characteristic polynomial.

To an arrangement $\Hh$
one associates the corresponding Orlik-Solomon algebra
as follows. Consider the anticommutative algebra $\Ee$ generated by
the variables $z_1,\ldots,z_r$ corresponding to hyperplanes.
For any set $K=\{i_1,\ldots,i_{k}\}\subset K_0$ we consider the monomial
$z_K=z_{i_1}\wedge\cdots\wedge z_{i_{k}}\in \Ee$.
 We can equip $\Ee$ with the natural differential $\partial$ sending $z_i$ to 1, namely
$$\partial (z_K)=\sum_{j=1}^{k}(-1)^{j-1}z_{K\setminus \{i_{j}\}}.$$
The natural degree of $z_K$ is $|K|$. Hence $\partial $ has degree $-1$.
\begin{definition}
We call the set $K$ {\it dependent}, if the linear equations of the corresponding hyperplanes are linearly dependent.
Otherwise $K$ is called {\it independent}.

The Orlik-Solomon ideal $\Ii$ is the ideal in $\Ee$ generated by the elements $\partial z_K$ for all dependent sets $K$.
The Orlik-Solomon algebra  is the quotient $\Aa=\Ee/\Ii$.
\end{definition}

\begin{theorem}(\cite[Theorem 5.2]{or1})
\label{osb}
The integral cohomology ring of the complement $V\setminus \cup_{i=1}^{r}\Hh_i$ is isomorphic to the Orlik-Solomon algebra $\Ee/\Ii$.
It has no torsion, and its Poincar\'e polynomial is given by the formula
$$P(\Hh,t)=(-t)^{\rho(K_0)}\cdot \chi_{\Hh}(-t^{-1})=\sum_{K\subset K_0}(-1)^{|K|}(-t)^{\rho(K)}.$$
\end{theorem}

As a corollary, we conclude that the homology of $V\setminus\cup_{i=1}^{r}\Hh_i$
is defined by its class in the Grothendieck ring, cf. Lemma \ref{lem:L}.
The same is true for its projectivization (see below). (This property of hyperplane complements
explain why the coefficients of the motivic Poincar\'e series can guide the complete cohomological
information.)

Later we will define a distinguished homological degree in $\Ee$, such that the above isomorphism will
preserve the corresponding gradings.

\vspace{2mm}

First, we consider the following `{\it deformation of the differential  on $\Ee$}'.

\begin{definition}
Let us define the following operator:
$$\partial_{U}:\Ee[U]\to \Ee[U],\quad  \partial_{U}(z_K)=\sum_{j=1}^{k}(-1)^{j-1}
U^{\rho(K)-\rho(K\setminus \{i_j\})}z_{K\setminus i_j},$$
where $U$ is a formal variable and $K=\{i_1,\ldots,i_{k}\}$.
\end{definition}

Note that $\rho(K)-\rho(K\setminus \{i_j\})\in \{0,1\}$, hence $\partial_U$ decomposes  into a sum of two components
\begin{equation}
\partial_U=\partial_0+U\partial_1,\quad \mbox{\rm with}\quad \partial_0+\partial_1=\partial.
\end{equation}

\begin{lemma}
\label{dd}
The operator $\partial_{U}$ is a differential on $\Ee[U]$, that is,
$\partial_{U}^2=0$.
In particular,  the following identities hold:
$$\partial_0^2=\partial_1^2=0, \ \ \ \partial_0\partial_1+\partial_1\partial_0=0.$$
\end{lemma}
\begin{proof} Straightforward.
\end{proof}

Let
 $\Jj$ and   $\Jj^{\perp}$ denote the {\it subspaces} of $\Ee$ spanned by the elements $z_K$ for all dependent, respectively
 independent  subsets $K$.
Clearly $\Ee=\Jj\oplus \Jj^{\perp}.$

\begin{lemma}
\label{d0reduction}
The following statements hold:

\begin{itemize}

\item[(a)] (\cite[Lemma 2.7]{or1},  \cite[Lemma 3.15]{or2}) \ \  $\Ii=\Jj+\partial \Jj$.

\item[(b)]$\partial_{0}\Jj^{\perp}=0$, hence $\Imm\partial_0=\partial_0\Jj$.

\item[(c)]$\partial_{1}\Jj\subset \Jj$, hence $\Ii=\Jj+\partial\Jj=\Jj+\partial_0\Jj$.

\item[(d)]$\ker\partial_0=\Jj^{\perp}+\Imm\partial_0$.

\item[(e)]There exist subspaces $A\subset \Jj,B\subset \Jj^{\perp}$ such that $\Imm\partial_0=A\oplus B.$

\end{itemize}
\end{lemma}

\begin{proof}
The claims (b) and (c) are clear. Let us prove (d).
The inclusion $\Jj^{\perp}+\Imm\partial_0\subset \ker\partial_0$ is also clear, hence
we need to prove that if $\partial_{0}(\phi)=0$ then there exists $\widetilde{\phi}\in \Jj^{\perp}$ such that
$\phi- \widetilde{\phi}\in \Imm(\partial_{0}).$

Let us call $z_i$ {\it essential} in a monomial $z_i\wedge z_{K}$, if $\rho(\{i\}\sqcup K)=\rho(K)+1$, and {\it redundant} otherwise.
Let us decompose $\phi=z_1\wedge \phi_1+z_1\wedge \phi_2+\phi_3$, where $z_1$ is essential in every monomial of $z_1\wedge \phi_1$,
redundant  in every monomial of $z_1\wedge \phi_2$, and $\phi_3$ contains no $z_1$.
Then
$$0=\partial_{0}(\phi)=z_1\wedge \psi+\phi_2+\partial_{0}(\phi_3)$$
for some $\psi$, and neither $\phi_2$ nor $\partial_{0}(\phi_3)$ contain $z_1$.
Hence $\phi_2=-\partial_{0}(\phi_3)$.
Since $z_1$ is redundant in every monomial in $z_1\wedge \partial_{0}(\phi_3)$, it is redundant in every monomial in $z_1\wedge \phi_3$ too.
Therefore
$$\partial_{0}(z_1\wedge \phi_3)=\phi_3-z_1\wedge \partial_{0}(\phi_3)+z_1\wedge \eta,$$
where $z_1$ is essential in every monomial of $z_1\wedge \eta$.
Indeed, if $i_{j}\in K$, $z_{i_j}$ is redundant in $K\cup \{1\}$ and essential in $K$, then $z_1$ is essential
in  $K\cup \{1\}\setminus \{i_j\}.$
We conclude that
$$\phi-\partial_0(z_1\wedge \phi_3)=z_1\wedge (\phi_1-\eta)$$
and $z_1$ is essential in every monomial in the right hand side.
Now,  $0=\partial_0(\phi)=-z_1\wedge \partial_0(\phi_1-\eta)$, hence  $\partial_0(\phi_1-\eta)=0$.
 Then we can repeat the procedure inductively replacing $\phi$ by $\phi_1-\eta$, and $z_1$ by $z_2$,  etc.
At the end we reduce $\phi$ modulo $\Imm(\partial_{0})$ to an element of $\Ee$
where all $z_{i}$ are essential; such an element belongs to $\Jj^{\perp}$.

Next, we  prove (e). Recall that $\Imm\partial_0=\partial_0\Jj$ and
$K$ is dependent iff $\rho(K)<|K|$. If the monomial $z_{K'}$ appears in $\partial_{0}(z_{K})$
then $\rho(K)=\rho(K')$ and $|K'|=|K|-1$. Therefore,
with $K$ dependent,  $\partial_{0}(z_{K})\in \Jj^{\perp}$
if $\rho(K)=|K|-1$, and $\partial_{0}(z_{K})\in \Jj$ otherwise.
\end{proof}

\begin{lemma}(cf. \cite[Lemma 3.42]{or2}, \cite[1.46]{BOOK09})
\label{d1 acyclic on A}
Let $\partial_1^{{\Aa}}$  be the differential induced by $\partial_1$
on $\Aa=\Ee/\Ii$.  Then $\partial_1^{{\Aa}}$ is acyclic,
 that is, ${\rm im }\,\partial_1^{{\Aa}}=\ker \partial_1^{{\Aa}}$.
\end{lemma}

\begin{proof} In the proof we always refer to the points (a)--(e) of Lemma \ref{d0reduction}.
Suppose that the class $[\alpha]\in \Aa=\Ee/\Ii$ belongs to the kernel of $\partial_1^{{\Aa}}$,
so $\partial_1(\alpha)\in \Ii$. By (c) we can assume that $\alpha\in \Jj^{\perp}$.  Then
$\partial(\alpha)=\partial_1(\alpha)\in \Ii\cap \Jj^{\perp}$. By (c)-(e)
$\Ii\cap \Jj^{\perp}=(\Jj+A\oplus B)\cap \Jj^{\perp}=B\subset \partial_0\Jj$, hence
 there exists $\alpha_1\in \Jj$ such that $\partial_1(\alpha)=\partial_0(\alpha_1)$.
Furthermore,
$\partial_0\partial_1(\alpha_1)=\partial_1\partial_0(\alpha_1)=0,$
hence $\partial_1(\alpha_1)\in \ker\partial_0\cap \partial_1\Jj$.
But again by (c)-(d)-(e) one has $\ker\partial_0\cap \partial_1\Jj\subset (\Jj^{\perp}+A\oplus B)\cap\Jj=A\subset \partial_0\Jj$.
Hence there exists $\alpha_2\in \Jj$ with $\partial_1(\alpha_1)=\partial_0(\alpha_2)$. Again,
$\partial_1(\alpha_2)\in \ker\partial_0\cap \partial_1\Jj$.
This procedure can be repeated to provide  $\alpha_3\in\Jj$ with $\partial_1(\alpha_2)=\partial_0(\alpha_3)$, and, in fact,
a sequence $\alpha_i\in \Jj$ with $\partial_0(\alpha_i)=\partial_1(\alpha_{i-1})$ ($\alpha_0=\alpha$).

Note that $\rho(\alpha_i)=\rho(\alpha)-i$, so  this process eventually stops. Now
$$
\partial(\alpha-\alpha_1+\alpha_2-\ldots)=\partial_1(\alpha)-\partial_0(\alpha_1)-\partial_1(\alpha_1)+
\partial_0(\alpha_2)+\partial_1(\alpha_2)-\ldots=0.
$$
Since $\partial$ is acyclic on $\Ee$, there exists $\beta$ such that
$
\partial(\beta)=\alpha-\alpha_1+\alpha_2-\ldots
$.
Let us decompose $\beta=\beta'+\beta''$, where $\beta'\in \Jj^{\perp}$ and $\beta''\in \Jj$, then by (b),
$$
\alpha=\partial_1(\beta')+\partial(\beta'')+\alpha_1-\alpha_2+\ldots\equiv \partial_1(\beta')\mod \Ii,
$$
hence $[\alpha]$ belongs to the image of  $\partial_1^{{\Aa}}$.
\end{proof}

The following theorem determine the homology of the complexes $(\Ee,\partial_0)$ and
$(\Ee[U],\partial_U)$.

\begin{theorem}
\label{homology of d0} \
(1) The homology of the differential $\partial_{0}$ is isomorphic to the
Orlik-Solomon algebra $\Aa=\Ee/\Ii$. This fact together with Theorem \ref{osb} provide
$$H_{*}(\Ee,\partial_{0})\simeq\Aa\simeq
H^{*}\left(V\setminus \cup_{i=1}^{r}\Hh_i\right).$$

(2) The homology of the differential $\partial_{U}$ is isomorphic (as $\BZ$--module) to the homology of the projectivized arrangement:
 $$H_{*}(\Ee[U],\partial_{U})\simeq
\ker\partial_{1}^{{\Aa}}\simeq H^{*}\left(\BP V\setminus \cup_{i=1}^{r}\BP \Hh_i\right),$$
and
it can be generated by a set of elements of type   $U^mz_K$, with $m=0$ and $K$ independent.

In particular, the induced  $U$--action on $H_{*}(\Ee[U],\partial_{U})$ is  trivial.

(3)  $\Ee$ is bi-graded: one can assign $|K|$, respectively $\rho(K)$, to $z_{K}$.
 $\partial_0$ decreases the first grading by $1$ and preserves the second one,
 hence $H_*(\Ee,\partial_0)$ is bi-graded too.  Nevertheless,
  the two gradings on $H_*(\Ee,\partial_0)$ agree, and the isomorphisms from (1) and (3)
  are graded isomorphisms (where  $H^{*}\left(V\setminus \cup_{i=1}^{r}\Hh_i\right)$
  and $H^{*}\left(\BP V\setminus \cup_{i=1}^{r}\BP \Hh_i\right)$ have their natural
  cohomological gradings).
\end{theorem}

\begin{proof}
(1) By Lemma \ref{d0reduction} one has
$\ker\partial_0=\Jj^{\perp}+\Imm \partial_{0}$ and $\Imm\partial_0=\partial_{0} \Jj$,
hence
$$H_{*}(\Ee,\partial_{0})=(\Jj^{\perp}+\partial_{0}\Jj)/\partial_{0}\Jj
\simeq \Jj^{\perp}/(\partial_{0}\Jj\cap \Jj^{\perp})\simeq \Ee/(\Jj+\partial_0\Jj).$$
The last identity follows from the splitting in Lemma \ref{d0reduction}(e). Then use Lemma \ref{d0reduction}(c).

(2)
Since $\partial_{U}=\partial_0+U\partial_1$, there exists a spectral sequence starting with
$H_{*}(\Ee[U],\partial_0)$ and converging to $H_{*}(\Ee[U],\partial_{U})$.
The $E^1$ page is
$((H_{*}(\Ee[U],\partial_0),U\partial_{1}^{\Aa}))=(\Aa[U],U\partial_{1}^{\Aa})$, and
by Lemma \ref{d1 acyclic on A} the $E^2$ page has a form:
$$
H_{*}(\Ee[U],\partial_{U})\simeq H_{*}(\Aa[U],U\partial_{1}^{{\Aa}})\simeq \ker\partial_{1}^{{\Aa}}.
$$
Since this homology is concentrated in the lowest $U$-degree, all higher differentials vanish.
This shows the first isomorphism of (2). For the second one  see \cite[Theorem 1.50]{BOOK09}.

(3)
>From the proof of part (1) follows that
$H^{*}(\Ee,\partial_0)$ can be identified with a quotient of $\Jj^{\perp}$.
Since $\Jj^{\perp}$ is spanned by the independent monomials, the gradings induced by
$|K|$ and $\rho(K)$ coincide on $H^{*}(\Ee,\partial_0)$. The isomorphisms  from
the already cited \cite[Theorem 5.2]{or1} and
\cite[Theorem 1.50]{BOOK09} are compatible with  this grading.
\end{proof}
\begin{remark}\label{rem:BP} (Cf. \cite[Corollary 3.58]{or2})
Since  $V\setminus \cup_{i=1}^{r}\Hh_i=\BC^*\times (\BP V\setminus \cup_{i=1}^{r}\BP \Hh_i)$,
the Poincar\'e polynomials $P(\Hh,t)$ and  $P(\BP\Hh,t)$
 of the cohomologies of the complements of the linear and projective arrangements satisfies
$(1+t)\cdot P(\BP\Hh,t)=P(\Hh,t)$.
\end{remark}

\begin{example}
Consider the arrangement of $r$ lines through the origin in $V=\BC^2$.
Then
$$\partial_{U}(1)=0,\ ~\partial_{U}(z_i)=U,\ ~\partial_{U}(z_i\wedge z_j)=U(z_i-z_j)=
U\partial(z_i\wedge z_j),$$
$$\partial_{U}(z_{K})=\partial(z_{K})\quad \mbox{\rm for}\quad |K|\ge 3.$$
The homology of $\partial_{U}$ is spanned by $1, z_1-z_2, \ldots, z_1-z_r$.
On the other hand, $\BP V\setminus \BP\Hh$ is the complement to $r$ points in $\mathbb{CP}^1$,
 homotopically equivalent to the bouquet of $(r-1)$ circles.
\end{example}

\subsection{The Orlik-Solomon complex and the lattice complex for curve singularities}\

Consider a curve singularity $C$, the associated lattice complex
(cf. Section \ref{s:lattice}) and the collection of local hyperplane arrangements $\Hh(v)$ (cf. \ref{ss:ARR}).
We wish to compare the Orlik-Solomon complex $(\Ee[U],\partial_U)$ associated with the local
hyperplane arrangement $\Hh(v)$ and the lattice complex $(\gr_v\Ll^-,\gr_v\partial_U)$.

\begin{theorem}\label{th:Lat+OS} (a)  For any fixed $v$ one has an isomorphism
$$H_{-2h(v)-b}(\gr_v \Ll^-,\gr_v\partial_U)=H_b(\Ee[U],\partial_U).$$
In the left hand side the homological degree is the one defined in (\ref{eq:HOMDEG1}),
while in the right hand side is induced by $\deg(z_K)=|K|$, cf. \ref{ss:4.2}.
(This is a $\BZ$ module isomorphism; since $U$ acts on $H_*(\Ee[U],\partial_U)$ trivially,
cf. \ref{homology of d0}, it acts on $H_*(\gr_v \Ll^-,\gr_v\partial_U)$ trivially as well.)

(b) Assume that the $ H_{a,b}(\gr_v\Ll^-,\gr_v\partial_U)\not=0$, where $(a,b)$ is the bi-grading
introduced in \ref{rem:bigr}. Then $(a,b)=(-2h(v)-2|K|,|K|)$ for some $K$.

\end{theorem}
\begin{proof}
Define $\psi:\gr_v\Ll^-\to \Ee[U]$ by $\psi(U^m\Box(v,K))=U^mz_K$. One verifies that it
is an isomorphism, and
$\partial _U\circ \psi=\psi\circ
\gr_v\partial_U$. Hence induces an isomorphism at homological level too.
By Theorem \ref{homology of d0}(3) for the generators we can assume that $m=0$ and $\rho_v(K)=|K|$.
Then the homological degree of $\sq(v,K)$ is $\deg=-2h(v+e_K)+|K|=-2h(v)-2\rho_v(K)+|K|=
-2h(v)-|K|$, while the degree of $z_K$ is $|K|$.
For (b) note  that the bi-degree of such $\sq(v,K)$ is $(-2h(v)-2|K|,|K|)$.
\end{proof}

\subsection{Proof of Theorem \ref{zlat}}
\label{sec:pf}

Assume $v\not\in \Ss$ and
fix  $i\in K_0$ such that $h(v)=h(v+e_i)$  (cf. Lemma \ref{eq:semi}), hence
 $h(v+e_K)=h(v+e_K+e_{i})$
for any $K$ with $K\not\ni i$ by \ref{lem:h matroid}.
Let $\phi:\gr_v\Ll^-\to
\gr_v\Ll^-$ be defined by
$$\phi(\Box(v,K))=\left\{\begin{array}{ll}
\Box(v,K\cup i_0) & \mbox{if $i_0\not\in K$},\\
0 &  \mbox{if $i_0\in K$}. \end{array}\right.$$
Then $\phi$ realizes a homotopy between the identity and the zero map:
$\partial_U\,\phi+\phi\,\partial_U=id$ on $\gr_v\Ll^-$. Hence $H_*(\gr_v\Ll^-)=0$.
On the other hand, $\Hh(v)=\emptyset$, hence $H^*(\BP\Hh(v))=0$ too.

If $v\in \Ss$  then parts (2) and (3) follow from Theorems \ref{th:Lat+OS} and \ref{homology of d0}(2).
A possible second proof of part (1) is the following (for the first proof see \ref{ss:3.2}):
$$P^{\Ll^-}_v(t^{-1})  \stackrel{\ref{th:Lat+OS}}{=}
t^{2h(v)}P(\Ee[U],\partial_{U},t)
\stackrel{\ref{homology of d0}(3)}{=}
t^{2h(v)}P(\BP\Hh(v),t)
\stackrel{\ref{rem:BP}}{=}
\frac{t^{2h(v)}}{1+t}P(\Hh(v),t) $$
$$\stackrel{\ref{osb}}{=}
\frac{t^{2h(v)}}{1+t}\cdot \sum_{K\subset K_0}\ (-1)^{|K|}(-t)^{\rho_v(K)}=
\frac{(-t)^{h(v)}}{1+t}\cdot \sum_{K\subset K_0}\ (-1)^{|K|}(-t)^{h(v+e_{K})}. 
$$

\section{Heegaard Floer link homology for algebraic links}\label{s:L}

\subsection{}
We can now apply the results of the previous sections to the computation of the Heegaard Floer homology of
algebraic links using the following result.

\begin{theorem}(\cite{gn})
All algebraic links are $L$-space links.
\end{theorem}

\begin{proposition}
\label{g equals h}
If $L$ is an algebraic link then its $HFL$-weight function coincides with the Hilbert function $h(v)$ up to an additive constant.
\end{proposition}
\begin{proof} By Theorem \ref{Poincare vs Alexander} the Poincar\'e series coincides
with the Alexander polynomial, hence with the Euler characteristic of $HFL^-$. The statement now
 follows from Theorems
\ref{reconst} and \ref{th:NEWTH}.
\end{proof}

\begin{theorem}
\label{thm:alg l space}
Let $L$ be an algebraic $L$-space link corresponding to a plane curve singularity $C$.
Then the spectral sequence defined in Theorem \ref{thm: l space ss} collapses  at $E^2$ page for all $v$:
$$\HFL^{-}(L,v)\simeq \HL^{-}(L,v)=H_{*}(\gr_{v} \Ll^{-}_{C}) \ \ \ \ (\mbox{as graded $\BZ$ modules}).$$
\end{theorem}

\begin{proof}
By Proposition
 \ref{g equals h} the $HFL$-weight function for $L$ coincides with the Hilbert function of $C$.
Consider the spectral sequence of Theorem \ref{thm: l space ss}. Its $E^2$ page coincides with
$H_*(\gr \Ll^{-})$.  We consider the bi-grading on $H_*(\gr_v\Ll^-)$, cf. \ref{rem:bigr}, and we use
the notations of the proof of Theorem \ref{thm: l space ss}. Note that the bi-grading
$(a,b)$ coincides exactly with $(\nu,|K|)$. Hence, by Theorems \ref{th:Lat+OS}
on the $E^2$ page all the non-trivial entries are on the line $\nu+2|K|+2h(v)=0$, while
the differential $d_k$ has bi-degree $(k-1,-k)$, hence two elements of this line are never
connected by  $d_k$
whenever $k\geq 2$. Hence $d_k=0$. \end{proof}

\begin{remark}
A similar spectral sequence was defined in the context of the subspace
arrangements by Jewell \cite{jewell}, who also proved its degeneration  at $E^2$ page.
\end{remark}

\begin{corollary}
By Corollary \ref{support hl}, the set of $v$ such that $\HFL^{-}(L,v)\neq 0$ coincides with the semigroup of $C$. In particular, the support of $\HFL^-$
determines the topological type of the algebraic link completely.
\end{corollary}

It is well known \cite{os} that for $L$-space {\it knots} (hence for all algebraic knots) the dimension
of the Heegaard Floer homology with given
Alexander grading is at most 1. For algebraic {\it links} we get the following generalization of this result (it was independently proven in \cite[Theorem 1.15]{Liu}
for general $L$-space links).

\begin{corollary}
If $L$ is an algebraic link, then $\rank \HFL^{-}(L,v)\le 2^{r-1}$ for all $v\in \BZ^r$.
For $v$ large enough ($v\succeq l$ in the notations of section \ref{ss:PA}) $\rank \HFL^{-}(L,v)=2^{r-1}$.
\end{corollary}

\begin{proof}
It is clear   from Theorem \ref{homology of d0} that the total dimension of the homology of the complement to $r$ hyperplanes cannot exceed $2^r$
and equals $2^r$ if and only if the hyperplanes are independent. By the same theorem, projectivization of the arrangement halves the total dimension of its homology.
It remains to note that by \eqref{eq:NEWNEW} the hyperplanes in the local arrangement $\Hh(v)$ are independent for $v\succeq l$.
\end{proof}

\section{Example. The Hopf link}\label{NewExample}

\subsection{}
We illustrate the main results of the paper for the positive Hopf link,
the  link of the $A_1$ singularity $\{xy=0\}$.
Its Alexander polynomial  equals $\Delta(t_1,t_2)=1$.

{\noindent \bf A. Hyperplane arrangements.} Let us describe the spaces $\Hh(v)$ explicitly. A function $g\in \BC[x,y]$ has order $0$ on one of the components if and only its constant term is nonzero,
and hence its order on the second component also equals $0$. Therefore
$$
\Hh(0,0)=\{\alpha+\text{higher order terms}\,|\,\alpha\neq 0\}\sim \BC^{*},\quad \Hh(a,0)=
\Hh(0,a)=\emptyset\ \text{for}\ \ a>0.
$$
Furthermore, for $a,b>0$ the space $\Hh(a,b)$ is
$$
\{\alpha x^{b}+\beta y^{a}+\text{terms of type $\gamma x^iy^j$ with
$(i,j)\geq (1,1)$, or $(b+1,0)$, or $(0,a+1)$}\,|\,\alpha,\beta\neq 0\},
$$
Therefore  $ \Hh(a,b)\sim (\BC^{*})^2$, and
\begin{equation}
\label{homology bp}
H^{*}(\BP\Hh(a,b))=\begin{cases}
H^*(\mbox{point})=\BZ & \text{if}\ a=b=0,\\
0 & \text{if}\ ab=0,\ (a,b)\not=(0,0)\\
H^*(\BC^*)=\BZ\oplus \BZ & \text{if}\ a,b>0.\\
\end{cases}
\end{equation}
Note that for $a,b>0$ the Euler characteristic of $\BP\Hh(a,b))$ vanishes, so
$$
\sum_{a,b\in \BZ^2}t_1^{a}t_2^{b}\ \chi(\BP\Hh(a,b)))=1=\Delta(t_1,t_2).
$$
{\noindent \bf B. Local lattice homology.} The Hilbert function of the $A_1$ singularity
is  (cf. Example \ref{ex:AN}):
$$
h(a,b)=\begin{cases}
\max(a,b),& \mbox{\rm if } \min(a,b)= 0,\\
a+b-1, & \mbox{\rm  otherwise.}\\
\end{cases}
$$
It is shown in Figure \ref{a1hilb}.
Let us compute the local lattice homology with the weight $h(v)$. 
\begin{figure}
\centering
\begin{tikzpicture}
\draw (0,0) node {\bf 0};
\draw (1,0) node {1};
\draw (2,0) node {2};
\draw (3,0) node {3};
\draw (4,0) node {4};
\draw (0,1) node {1};
\draw (1,1) node {\bf 1};
\draw (2,1) node {\bf 2};
\draw (3,1) node {\bf 3};
\draw (4,1) node {\bf 4};
\draw (0,2) node {2};
\draw (1,2) node {\bf 2};
\draw (2,2) node {\bf 3};
\draw (3,2) node {\bf 4};
\draw (4,2) node {\bf 5};
\draw (0,3) node {3};
\draw (1,3) node {\bf 3};
\draw (2,3) node {\bf 4};
\draw (3,3) node {\bf 5};
\draw (4,3) node {\bf 6};
\draw (0,4) node {4};
\draw (1,4) node {\bf 4};
\draw (2,4) node {\bf 5};
\draw (3,4) node {\bf 6};
\draw (4,4) node {\bf 7};

\draw [dashed] (0.2,0)--(0.9,0);
\draw [dashed] (1.2,0)--(1.9,0);
\draw [dashed] (2.2,0)--(2.9,0);
\draw [dashed] (3.2,0)--(3.9,0);
\draw [dashed] (0.2,1)--(0.9,1);
\draw [dashed] (1.2,1)--(1.9,1);
\draw [dashed] (2.2,1)--(2.9,1);
\draw [dashed] (3.2,1)--(3.9,1);
\draw [dashed] (0.2,2)--(0.9,2);
\draw [dashed] (1.2,2)--(1.9,2);
\draw [dashed] (2.2,2)--(2.9,2);
\draw [dashed] (3.2,2)--(3.9,2);
\draw [dashed] (0.2,3)--(0.9,3);
\draw [dashed] (1.2,3)--(1.9,3);
\draw [dashed] (2.2,3)--(2.9,3);
\draw [dashed] (3.2,3)--(3.9,3);
\draw [dashed] (0.2,4)--(0.9,4);
\draw [dashed] (1.2,4)--(1.9,4);
\draw [dashed] (2.2,4)--(2.9,4);
\draw [dashed] (3.2,4)--(3.9,4);
\draw [dashed] (0,0.3)--(0,0.8);
\draw [dashed] (1,0.3)--(1,0.8);
\draw [dashed] (2,0.3)--(2,0.8);
\draw [dashed] (3,0.3)--(3,0.8);
\draw [dashed] (0,1.3)--(0,1.8);
\draw [dashed] (1,1.3)--(1,1.8);
\draw [dashed] (2,1.3)--(2,1.8);
\draw [dashed] (3,1.3)--(3,1.8);
\draw [dashed] (0,2.3)--(0,2.8);
\draw [dashed] (1,2.3)--(1,2.8);
\draw [dashed] (2,2.3)--(2,2.8);
\draw [dashed] (3,2.3)--(3,2.8);
\draw [dashed] (0,3.3)--(0,3.8);
\draw [dashed] (1,3.3)--(1,3.8);
\draw [dashed] (2,3.3)--(2,3.8);
\draw [dashed] (3,3.3)--(3,3.8);
\draw [dashed] (4,0.3)--(4,0.8);
\draw [dashed] (4,1.3)--(4,1.8);
\draw [dashed] (4,2.3)--(4,2.8);
\draw [dashed] (4,3.3)--(4,3.8);

\end{tikzpicture}
\caption{Values of the Hilbert function for $A_1$ singularity}
\label{a1hilb}
\end{figure}
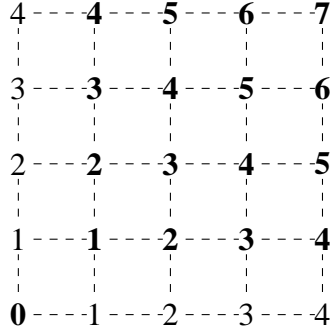
For all $v=(a,b)$ the local lattice complex has 4 generators $a_v,\alpha_v,\beta_v,\Gamma_v$ over $\BZ[U]$.
Here $a_v$ can be identified with the point $v$, $\alpha_v$ and $\beta_v$ can be identified with the east- and northward pointing segments starting at $v$
 and $\Gamma_v$ can be identified with the square with minimal vertex $v$. 
 The differential is given by the equation:
$$
\partial(a_v)=0,\ \partial(\alpha_v)=U^{h(a+1,b)-h(a,b)}a_v,\ \partial(\beta_v)=U^{h(a,b+1)-h(a,b)}a_v,\ $$ 
$$\partial(\Gamma_v)=U^{h(a+1,b+1)-h(a+1,b)}\alpha_v-U^{h(a+1,b+1)-h(a,b+1)}\beta_v.
$$
For $v=(0,0)$ one has $\partial(\alpha_v)=\partial(\beta_v)=Ua_v, \partial(\Gamma_v)=\alpha_v-\beta_v,$,
 so the homology is spanned by $a_v$.
For $v=(a,0), a>0$
one has $\partial(\alpha_v)=Ua_v, \partial(\beta_v)=a_v, \partial(\Gamma_v)=\alpha_v-U \beta_v,$ and the homology vanishes
(similarly as  for $v=(0,a)$). Finally, for $v=(a,b), a,b>0$ one has
$\partial(\alpha_v)=\partial(\beta)=Ua_v,\ \partial(\Gamma_v)=U(\alpha_v-\beta_v)$ 
and the homology is spanned by $a_v$ and $\alpha_v-\beta_v$,
in agreement with \eqref{homology bp}. The homological degrees are $-2(a+b)+2$ and $-2(a+b)+1$.
Note that  $U$ acts by $0$ on the homology in all cases.

{\noindent \bf C. Link Floer homology.}
 Similarly to \cite[Section 12]{os3}, one can check that the minimal Heegaard Floer complex $CFL^{-}$ has four $\BZ[U_1,U_2]$-generators $\alpha,\beta,\gamma,\delta$ of Alexander
gradings  $(0,0)$, $(1,0)$, $(0,1)$, $(1,1)$ and homological degrees $0,-1,-1,-2$.
The differential is $\BZ[U_1,U_2]$--linear given by the formula:
$$
d(\beta)=U_1\alpha+\delta,\ d(\gamma)=U_2\alpha+\delta,\ d(\alpha)=d(\delta)=0.
$$
The filtered subcomplex $\Cc(v)$ is spanned by all elements of Alexander grading greater than or equal to $v$.
By definition, $\HFL^-(v)$ is the homology of the associated graded complexe $\gr \Cc(v)$. For $v=(0,0)$ the complex $\gr \Cc(0,0)$ is generated over $\BZ$ by a single element $\alpha$.
For $a>0$ the complex $\gr \Cc(a,0)$ is generated  over $\BZ$
by $U_1^{a}\alpha,  U_1^{a-1}\beta$, with the differential
$$d_{\gr}(U_1^{a-1}\beta)=U_1^{a}\alpha.$$
Therefore $\gr A^-(a,0)$  (and similarly $\gr A^-(0,a)$) is acyclic.
Finally, for $a,b>0$ the complex  $\gr \Cc(a,b)$ is generated  by $U_1^{a}U_2^{b}\alpha,$ $U_1^{a-1}U_2^{b}\beta$, $ U_1^{a}U_2^{b-1}\gamma$ and $U_1^{a-1}U_2^{b-1}\delta$, with the differential
$$d_{\gr}(U_1^{a-1}U_2^{b}\beta)=d_{\gr}(U_1^{a}U_2^{b-1}\gamma)=U_1^{a}U_2^{b}\alpha.$$
Its homology (in  agreement with \eqref{homology bp}) equals
$$\HFL^-(a,b)=H^{*}(\gr \Cc(a,b))\simeq \BZ\langle U_1^{a-1}U_2^{b-1}\delta, U_1^{a-1}U_2^{b}\beta-U_1^{a}U_2^{b-1}\gamma \rangle.$$
{\noindent \bf D. Filtered subcomplexes.}
Let us also compute the homology of $\Cc(v)$ for various $v$.
The complex $\Cc(0,0)$ coincides with $CFL^{-}$ and its homology has the form
$$H^{*}(\Cc(0,0))=\BZ[U_1,U_2]\langle \alpha \rangle/(U_1\alpha=U_2\alpha)\simeq \BZ[U]\langle \alpha \rangle.$$
For $a>0$ the complex $\Cc(a,0)$ is generated over $\BZ[U_1,U_2]$ by $U_1^{a}\alpha, U_1^{a-1}\beta, U_1^{a}\gamma$ and
$U_1^{a-1}\delta$. One can check that
$$H^{*}(\Cc(a,0))\simeq \BZ[U]\langle U_1^{a-1}\delta \rangle,$$
and its generator has homological degree $-2a$. Similarly,
$H^{*}(\Cc(0,b))\simeq \BZ[U]\langle U_2^{b-1}\delta \rangle$ generated at
 degree $-2b$.
Finally,  for $a,b>0$ the subcomplex $\Cc(a,b)$ is generated over $\BZ[U_1,U_2]$ by
$U_1^{a}U_2^{b}\alpha, U_1^{a-1}U_2^{b}\beta, U_1^{a}U_2^{b-1}\gamma$ and
$U_1^{a-1}U_2^{b-1}\delta$.
One can check that
$$H^{*}(\Cc(a,b))\simeq \BZ[U]\langle U_1^{a-1}U_2^{b-1}\delta \rangle,$$
and its generator has homological degree $-2a-2b+2$.
Therefore for all $v$ the subcomplex $\Cc(v)$ is a free $\BZ[U]$-module of rank 1, and its generator has homological degree $-2h(v)$.

\end{document}